\documentclass[p10]{siamltex}
\usepackage{amsmath, amssymb, latexsym, amscd, amssymb, mathrsfs,t1enc,amsfonts}
\usepackage{pstricks,pst-plot,pst-node,pst-grad,pst-3d, hhline, multirow, graphicx,array}
\newtheorem{thm}{Theorem}[section]

\newtheorem{remark}[thm]{Remark}
\numberwithin{equation}{section}

\def\bff{{\mathbf f}}
 \def\bn{{\mathbf n}}
 
\def\bu{{\mathbf u}} \def\bv{{\mathbf v}}

\def\bH{{\mathbf H}}

\def\bN{{\mathbf N}}
\def\bphi{\boldsymbol{\phi}}

\def\RR{{\mathbb R}}

\newcommand{\wtilde}{\widetilde}
\newcommand{\what}{\widehat}
\newcommand{\Div}{\mathrm{div}}

 \newcommand{\bsigma}{{\boldsymbol{\sigma}}}

\newcommand{\bepsilon}{\boldsymbol{\epsilon}}

\def\bphi{\mbox{\boldmath{$\phi$}}}
\def\bsigma{\mbox{\boldmath{$\sigma$}}}

\title{ Immersed finite element method for eigenvalue problems in elasticity }

\author{ {Seungwoo Lee}\footnotemark[1] \and {Do Y. Kwak}\footnotemark[1] 
        \and Imbo Sim \footnotemark[2]}
\begin{document}
\maketitle

\renewcommand{\thefootnote}{\fnsymbol{footnote}}
\footnotetext[1]{Department of Mathematical Science, Korea Advanced Institute of Science and Technology, 305-701 Daejeon, Republic of Korea.}
\footnotetext[2]{National Institute for Mathematical Sciences, 305-811 Daejeon, Republic of Korea \\({\tt imbosim@nims.re.kr}).}

%


\begin{abstract}
We consider the approximation of eigenvalue problems for elasticity equations with interface. This kind of problems can be efficiently discretized by using immersed finite element method (IFEM) based on Crouzeix-Raviart P1-nonconforming element. The stability and the optimal convergence of IFEM for solving eigenvalue problems with interface are proved by adapting spectral analysis methods for the classical eigenvalue problem. Numerical experiments demonstrate our theoretical results.
\end{abstract}


\pagestyle{myheadings}
\thispagestyle{plain}

\section{ Introduction }
\indent In this paper, we consider the approximation of eigenvalue problems with interface in elasticity. Eigenvalue analysis is essential basis for many types of engineering analysis. As eigenvalues are closely related with the frequency and shape of structures, computing the eigensolutions is important to interpret the dynamic interaction between the structures. If the frequency of structures is close to the system's natural frequency, mechanical resonance occurs. It may lead to catastrophic failure or damage in constructed structures such as bridges, buildings, and towers \cite{Green-Unruh}.

There have been mathematical studies of finite element methods for eigenvalue problems. In \cite{Trefethe-Betcke} various computed examples for Laplacian eigenproblems in planar regions are studied and there are references to physical problems where the results are relevant. For nonconforming approximation of elliptic eigenvalue problems, it is shown that the eigenvalues computed by finite element methods give lower bounds of the exact eigenvalues whose eigenfunctions are singular in non-convex polygon \cite{Armentano-Duran}. The guaranteed lower and upper bounds of eigenvalues based on the nonconforming finite element approximation are given in \cite{Carstensen-Gedicke}. Moreover, let us focus on eigenvalue problems in elasticity. A posteriori error estimator for linearized elasticity eigenvalue problems is studied in \cite{Walsh-Reese-Hetmaniuk}. It is shown that upper and lower estimates for the error of eigenpairs are established in terms of a residual estimate and lower-order terms. In \cite{Oden-Prudhomme}, a method for three-dimensional linear elasticity or shell problems is presented to derive computable estimates of the approximation error in eigenvalues. The spectral problem for the linear elasticity equations on curved non-convex domains, as well as with mixed boundary conditions is considered in \cite{Hernandez}. Meddahi et al. \cite{Meddahi-Mora-Rodri} present an analysis for the eigenvalue problem of linear elasticity by means of a mixed variational formulation. This method weakly imposes the symmetry of the stress tensor and is free from the locking phenomenon.

When elastic body is occupied by heterogeneous materials, it is known that governing equations contain the discontinuous material parameters along the interface of materials. To simulate such problems by finite element methods, a common strategy is to use fitted meshes along the interface. However, this strategy may require a very fine mesh near the interface. An alternative approach, proposed in \cite{Chang-Kwak,Chou-Kwak-Wee,Kwak-W-C,Li-Lin-Lin-Rogers,Li-Lin-Wu}, is an immersed finite element method (IFEM) which can use any meshes independent of interface geometry. The idea of an IFEM is to construct local basis functions to satisfy the interface conditions. For source problems with interface in elasticity, Kwak et al.\cite{Kwak-Jin-Kyeong} present a nonconforming IFEM based on the broken Crouzeix-Raviart (CR) element \cite{Crouzeix-Raviart}. They prove optimal error estimates and provide numerical results for compressible and nearly incompressible materials. Computation results of IFEM based on the rotated $Q_1$-nonconforming element are reported in \cite{Lin-Sheen-Zhang} and the related work in this direction can be found in \cite{Lin-Zhang}. In addition, the spectral analysis of IFEM for elliptic eigenvalue problems with an interface is given in \cite{Lee-Kwak-Sim}.

In this work, we analyze the spectral approximation of elasticity interface problems using $P_1$-nonconforming IFEM and derive the optimal convergence of eigenvalues. Moreover, we provide a series of numerical results of the eigenproblems with various shapes of interface for compressible and incompressible materials. As a model problem, we consider an elasticity eigenvalue problem where the domain is separated into two subdomains by interface. The elastic modulus of the material in each subdomain is discontinuous along the interface and the eigenfunctions must satisfy certain interface conditions. We construct local basis functions to satisfy the jump conditions across the interface. Also our local basis functions are based on CR element. It is known that CR element does not lock on pure displacement problems \cite{Brenner_Sung}. For a traction boundary problem, the discrete scheme with a stabilization term is introduced to overcome locking \cite{P.Hansbo_Lar2002}. Since interface conditions are related to traction conditions, IFEM based on CR elements does not suffer the effects of locking by introducing the stabilization term. Furthermore, optimal orders of convergence in the $H^1$ and $L^2$-norms for IFEM are proved in \cite{Kwak-Jin-Kyeong}. Exploiting the ideas of \cite{Kwak-Jin-Kyeong} we formulate the discrete scheme with a stabilization term. The proofs for the spectral correctness of IFEM are based on the analysis of \cite{Babuska-Osborn,Descloux-Nassif-Rappaz1978-1,Descloux-Nassif-Rappaz1978-2,Osborn}. Introducing a solution operator, we use spectral properties of compact and self-adjoint operators in Banach space \cite{Alonso-Russo,Antonietti-Buffa-Perugia,Beattie,Buffa-Perugia}. In our analysis we adapt the approximation properties of IFEM from \cite{Kwak-Jin-Kyeong} to establish the spectral analysis of IFEM. Our proofs for such spectral approximation are very similar to the proofs of \cite{Lee-Kwak-Sim} which introduced IFEM to an elliptic eigenvalue problem with an interface.

The outline of this paper is as follows. In the next section, we give a description of elasticity eigenvalue problems with interface. In Section 3, we introduce a local basis function satisfying interface conditions and formulate an immersed finite element method with a stabilization term. Section 4 is devoted to the analysis of the spectral approximation which is proved to be spurious-free. In Section 5, we carry out numerical experiments for our model problem. The results demonstrate spurious-free and locking-free character of IFEM.

\section{Model problem}

Let $\Omega$ be a connected and convex polygonal domain in $\RR^2$ which is divided into two subdomains $\Omega^+$ and $\Omega^-$ by a $C^2$ interface $\Gamma = \partial \Omega^+ \cap \partial\Omega^- $ (see Figure \ref{fig:domain1}). We assume that the subdomains $\Omega^+$ and $\Omega^-$ are occupied by two different elastic materials.
Let $\lambda$ and $\mu$ denote the Lam\'{e} coefficients given by
$$\lambda = \frac{E \nu}{(1+\nu)(1-2\nu)},\,\, \mu = \frac{E}{2(1+\nu)},$$
where $E$ is the Young's modulus and $\nu$ is the Poisson ratio. We note that the coefficients $\lambda$ and $\mu$ are $0<\mu_1 <\mu <\mu_2$ and $0<\lambda<\infty$. The constitutive equation is related to the displacement field $\bu := (u_{i})\in \RR^2$ and the Cauchy stress tensor $\bsigma := (\sigma_{ij}) \in \RR^{2\times2}$ is given by
$$\bsigma(\bu) = 2\mu\, \bepsilon(\bu) + \lambda\, tr(\bepsilon(\bu))\boldsymbol{I},$$
where $\boldsymbol{I}$ is the identity matrix of $\RR^{2\times2}$, the linearized strain tensor $\bepsilon := (\epsilon_{ij}) \in \RR^{2\times2}$ is
$$\bepsilon(\bu) = \frac{1}{2}(\nabla \bu + {\nabla \bu}^T),$$
and the usual trace operator $tr(\bepsilon)$ is
$$ tr(\bepsilon) := \sum^{2}_{i=1} \epsilon_{ii}.$$
For the sake of simplicity, we assume that the density $\rho$ is a positive piecewise constant in subdomains $\Omega^+$ and $\Omega^-$. From now on, we consider the Lam\'{e} coefficients $\lambda$ and $\mu$ as $\lambda := \lambda / \rho$ and $\mu := \mu / \rho$. Let us consider the eigenvalue problem for the linear elasticity equation with interface, i.e.
\begin{eqnarray}\label{eq:Model}
-\Div\, \bsigma(\bu)  &=& \omega^2\bu \quad \mathrm{   in}~  \Omega^s \quad (s=+,-), \\
{[\bu]}_\Gamma &=& 0,  \label{jump2.2} \\
{[\bsigma(\bu)\cdot \bn]}_\Gamma&=&0, \label{jump2.3} \\
\bu &=& 0  \quad\quad\; \mathrm{ on} \;  \partial\Omega, \nonumber
\end{eqnarray}
where $\omega^2$ and $\bu$ are the corresponding eigenvalue and eigenfunction, and the symbol $[\cdot]$ denotes the jump across the interface $\Gamma$.

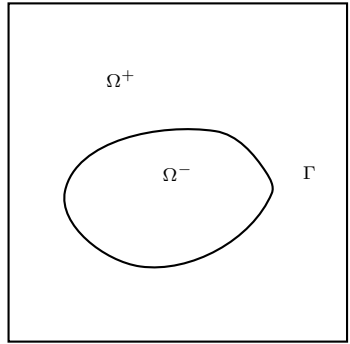
\begin{figure}[ht]
\begin{center}
      \psset{unit=2.5cm}
      \begin{pspicture}(-1,-1)(1,1)
        \pspolygon(0.9,0.9)(-0.9,0.9)(-0.9,-0.9)(0.9,-0.9)
        \psccurve(0.47,0) (0.2,0.22)(-0.6,-0.1)(-0.2,-0.5)(0.5,-0.11)
        \rput(0,0){\scriptsize$\Omega^-$}
        \rput(-0.3,0.5){\scriptsize$\Omega^+$}
        \rput(0.7,0){\scriptsize$\Gamma$}
      \end{pspicture}
\caption{A domain $\Omega$ with interface} \label{fig:domain1}
\end{center}
\end{figure}

 We formulate the model problem (\ref{eq:Model}) into the displacement formulation \cite{Braess}. Multiplying $\bv \in (H^1_0 (\Omega))^2$ and applying Green's identity to model problem (\ref{eq:Model}) in each domain $\Omega^s$, we obtain
\begin{equation*} \label{variational_form}
\int_{\Omega^s} 2 \mu\, \bepsilon(\bu):\bepsilon(\bv)dx + \int_{\Omega^s} \lambda\, \Div\, \bu\,\Div\, \bv\,dx - \int_{\partial\Omega^s}\bsigma(\bu) \mathbf{n} \cdot \bv ds = \omega^2\int_{\Omega^s} \bu \cdot \bv dx,
\end{equation*}
where 
\begin{equation*}
\bepsilon(\bu):\bepsilon(\bv) = \sum_{i,j=1}^2 \epsilon_{ij}(\bu) \epsilon_{ij} (\bv).
\end{equation*}
Summing over $s=+,-$ and applying the interface condition (\ref{jump2.3}), we have the following weak formulation
\begin{equation} \label{eq:variational_form}
a(\bu,\bv) =  \omega^2(\bu,\bv),
\end{equation}
where
\begin{equation*} \label{a_form}
a(\bu,\bv) = \int_{\Omega} 2 \mu\, \bepsilon(\bu):\bepsilon(\bv)dx + \int_{\Omega} \lambda\, \Div\, \bu\,\Div\, \bv\,dx
\end{equation*} and
\begin{equation*}
\omega^2(\bu,\bv)= \omega^2\int_{\Omega} \bu \cdot \bv dx.
\end{equation*}

\section{Immersed finite element method}
In this section, we introduce an immersed finite element method (IFEM) based  on Crouzeix-Raviart elements \cite{Crouzeix-Raviart}. Let $\{\mathcal{K}_h\}$ be the usual quasi-uniform triangulations of the domain $\Omega$ by the triangles of maximum diameter $h$. Note that an element $K \in \mathcal K_h$ is not necessarily aligned with the interface $\Gamma$. For a smooth interface, provided that $h$ is sufficiently small, we are able to assume that the interface intersects the edge of an element at no more than two points and joins each edge at most once, except possibly it passes through two vertices. We may replace $\Gamma\cap K$ by the line segment joining two intersection points on the edges of each $K\in \mathcal K_h$. We call an element $K\in\mathcal{K}_h$ an \textit{interface element} if the interface $\Gamma$ passes through the interior of $K$, otherwise $K$ is a \textit{non-interface element}. Additionally we introduce some symbols:
\renewcommand{\labelitemi}{$\cdot$}
\begin{itemize}
\item $\mathcal K_h^*$ - the collection of all interface elements
\item $\mathcal{E}_h\;$  - the collection of all the edges of $K\in \mathcal K_h$
\end{itemize}
 We are going to construct local basis functions on each element $K$ of the triangulation  $\mathcal K_h$. For a non-interface element $K\in \mathcal K_h$, we choose a standard $P_1$-nonconforming basis whose degrees of freedom are determined by average values on each edge of an element $K$. Let $\bN_h(K)$ denote the linear space spanned by the six Lagrange basis functions
$$\bphi_i = (\phi_{i1},\phi_{i2})^T,\ i=1,2,\cdots,6, $$
satisfying
\begin{eqnarray*}
\frac{1}{|e_j|}\int_{e_j}\phi_{i1}\, ds &=& \delta_{ij},\\
\frac{1}{|e_j|}\int_{e_j}\phi_{i2}\, ds &=& \delta_{i-3,j},
\end{eqnarray*}
for each edge $e_j$ of an element $K$, $j=1,2,3$. The $P_1$-nonconforming space $\bN_h(\Omega)$ is given by
$$ \bN_h(\Omega)= \left\{
\begin{aligned}
 &\bphi = (\phi_1,\phi_2)|_K\in \bN_h(K)\mbox{ for each } K\in\mathcal K_h \setminus \mathcal K_h^* ; \\
 & \mbox{if $K_1,K_2 \in \mathcal K _h$ share an edge $e$,}\mbox{ then, for $i=1,2$,}\\
 &\int_{e}{\phi_i}|_{\partial K_1} ds= \int_{e}{\phi_i}|_{\partial K_2}  ds; \mbox{ and }
  \int_{\partial K \cap \partial\Omega}{\phi_i}\,ds=0
\end{aligned}
\right\}.$$

\begin{figure}[ht]
  \begin{center}
    \psset{unit=3.0cm}
    \begin{pspicture}(0,0)(1,1)
      \psset{linecolor=black} \pspolygon(0,0)(1,0)(0,1) \psline(0,0.65)(0.35,0)
      \pscurve(0,0.65)(0.1,0.58)(0.2,0.15)(0.35,0)
      \rput(0,1.05){\scriptsize$A_3$}
      \rput(-0.05,0){\scriptsize$A_1$}
      \rput(1.05,0){\scriptsize$A_2$}
      \rput(0.5,-0.13){$e_3$}
      \rput(0.55,0.55){$e_1$}
      \rput(-0.06,0.5){$e_2$}
      \rput(-0.06,0.65){\scriptsize$E$}
      \rput(0.08,0.12){\scriptsize$K^-$}
      \rput(0.45,0.25){\scriptsize$K^+$}
      \rput(0.35,-0.05){\scriptsize$D$}
      \rput(0.20,0.48){\scriptsize$\Gamma$}
\pnode(-.3,0.6){a}
\pnode(0.12,0.5){b}
%
    \end{pspicture}
    \caption{A typical interface triangle} \label{fig:interel}
\end{center}
\end{figure}
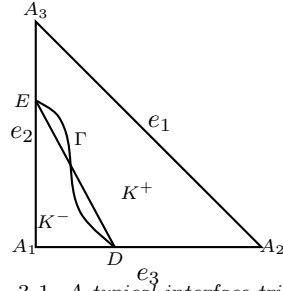
For an interface element $K \in \mathcal K_h$ (see Figure \ref{fig:interel}), we describe how to construct the basis functions which satisfy the interface conditions (\ref{jump2.2}), (\ref{jump2.3}). The piecewise linear basis function $\what \bphi_i$, $i=1,2,\cdots,6 $, of the form
\begin{eqnarray*}
 &&\what\bphi_i(x,y) = \left\{%
 \begin{array}{ll}
    \what\bphi^+_i (x,y) =
    \begin{pmatrix}
     \hat\phi^+_{i1} \\  \hat\phi^+_{i2}
    \end{pmatrix} = \begin{pmatrix} a_0^+ + b_0^+ x + c_0^+ y \\
    a_1^+ + b_1^+ x + c_1^+ y \end{pmatrix}, \quad (x,y) \in T^+,  \\
     &\\
    \what\bphi^-_i (x,y) =\begin{pmatrix}
    \hat \phi^-_{i1} \\  \hat\phi^-_{i2}
    \end{pmatrix} = \begin{pmatrix} a_0^- + b_0^- x + c_0^- y  \\
    a_1^- + b_1^- x + c_1^- y \end{pmatrix}, \quad (x,y) \in T^-, \\
\end{array}%
\right. \label {def:basis-1}
\end{eqnarray*}  satisfies

\begin{eqnarray*}\label{eq:dof1}
\frac{1}{|e_j|}\int_{e_j} \hat\phi_{i1}\,ds & = & \delta_{ij},\,\, j=1,2,3,\\
\frac{1}{|e_j|}\int_{e_j} \hat\phi_{i2}\,ds & = & \delta_{(i-3)j},\,\, j=1,2,3,\label{eq:dof2}\\
{[\what \bphi_i(D)]} & = & 0, \\
{[\what \bphi_i(E)]} & = & 0, \\
\left[\bsigma(\what\bphi_i) \cdot \mathbf{n}\right]_{\overline{\textrm{\tiny{$DE$}}}} & = & 0\label{eq:trac}.
\end{eqnarray*}
We can express these conditions as a square system of linear equations in twelve unknowns for each basis function $\what\bphi_i$. It is shown that this system has a unique solution regardless of the location of the interface (see \cite{Kwak-Jin-Kyeong}). Let us denote $\what{\bN}_h(K)$ as the space of functions on an interface element $K$, which is generated by $\what{\bphi}_i,\, i = 1,2,\cdots,6$. Using this local finite element space, we define the global {\em immersed finite element space} $\what{\bN}_h (\Omega)$
by
$$ \what{\bN}_h (\Omega)=
\left\{\begin{array}{l}
\what \bphi \in \what{\bN}_h(K) \mbox{ if } T\in \mathcal{K}_h^*, \mbox{ and }  \what \bphi \in \bN_h(K) \mbox{ if } K \not\in \mathcal{K}_h^*; \\
\mbox{ if $K_1$ and $K_2$ share an edge $e$, then $\what \bphi =(\hat\phi_1, \hat\phi_2)$ satisfies, }\\
  \int_{e}{\hat\phi_i}|_{\partial K_1} ds= \int_{e}{\hat\phi_i}|_{\partial K_2}  ds; \mbox{ and }
  \int_{\partial K \cap \partial\Omega}{\hat\phi_i}\,ds=0,\, (i=1,2)
 \end{array}
\right\}. $$

In order to describe analysis of IFEM, we introduce some spaces and their norms. For a bounded domain $D$ and non-negative integer $m$, we let $H^m(D) = W_2^m(D)$ be the usual Sobolev space of order $m$ with (semi)-norms denoted by $\|\cdot\|_{m,D}$ ($|\cdot|_{m,D}$) and let
\begin{eqnarray*}
(\wtilde{H}^m(D))^2 &:=& \{\,\bu\in (H^{m-1}(D))^2 : \,\bu|_{D\cap \Omega^s}\in
(H^m(D\cap \Omega^s))^2, s = +,-\,\},
\end{eqnarray*}
equipped with norms
\begin{eqnarray*}
|\bu|^2_{\wtilde{H}^m(D)} &:=& |\bu|^2_{m,D\cap \Omega^+} + |\bu|^2_{m,D\cap \Omega^-},\\
\|\bu\|^2_{\wtilde{H}^m(D)} &:=& \|\bu\|^2_{m,D\cap \Omega^+} + \|\bu\|^2_{m,D\cap \Omega^-}.
\end{eqnarray*}
In addition, we define the space $\bH_h(\Omega)$ by $\bH_h(\Omega) := (H^1_0(\Omega))^2 + \what \bN_h(\Omega)$.

The IFEM for the eigenvalue problem (\ref{eq:Model}) is to find the eigensolution $(\omega_h^2, \bu_h) \in \mathbb{C} \times \what{\bN}_h(\Omega)$ such that
\begin{equation}\label{eq:disc-form}
a_h(\bu_h,\bv_h) = \omega_h^2(\bu_h,\bv_h), ~\quad ~~ \forall \bv_h \in \what{\bN}_h (\Omega),
\end{equation}
where
\begin{align}\label{a_h_form}
a_h(\bu,\bv):&=\sum_{K\in \mathcal{K}_h}\int_K 2 \mu\, \bepsilon(\bu):\bepsilon(\bv)dx +
 \sum_{K\in \mathcal{K}_h}\int_K  \lambda\, \Div\, \bu\,\Div\, \bv\,dx  \\
& + \sum_{e \in \mathcal{E}_h} \frac{\tau}{h} \int_{e} [\bu][\bv]ds , \quad \forall \bu,\bv \in \bH_h(\Omega). \nonumber
\end{align}
The parameter $\tau$ in the bilinear form $a_h(\cdot,\cdot)$ is a positive constant which is independent of the mesh size $h$. We define the mesh dependent norm $\|\cdot\|_{a,h}$ on the space $\bH_h(\Omega)$ by
\begin{equation*} \label{Energy_norm}
\|\bv\|^2_{a,h}:= \sum_{K\in\mathcal{K}_h}\|\bv\|_{a,K}^2+\sum_{e\in \mathcal{E}_h} \int_{e} \frac{\tau}{h}[\bv]^2 ds,
\end{equation*}
where
\begin{equation*} \label{Energy_norm-T}
\|\bv\|_{a,K}^2=\int_{K} 2\mu\, \bepsilon(\bv):\bepsilon(\bv) dx +\int_{K} \lambda\, |\Div\, \bv|^2 dx.
\end{equation*}
\begin{remark}
The idea of the discrete scheme is motivated from Hansbo and Larson \cite{P.Hansbo_Lar2002}. For a source problem without an interface, they prove an optimal convergence of the scheme. For the problem with an interface, Kwak et al. \cite{Kwak-Jin-Kyeong} show the scheme yields an optimal result.
\end{remark}

The coerciveness and boundedness of the bilinear form $a_h(\cdot, \cdot)$ are satisfied \cite{Kwak-Jin-Kyeong}.
\begin{theorem}
There exist positive constants $C_b$ and $C_c$ such that
\begin{alignat*}{2}
|a_h (\bu,\bv)| &\leq C_b\|\bu\|_{a,h}\|\bv\|_{a,h}, \quad &\forall\, \bu,\bv \in \bH_h(\Omega), \\
a_h (\bv,\bv) &\geq  C_c\|\bv\|^2_{a,h}, 	 &\forall\, \bv \in \what \bN_h(\Omega).
\end{alignat*}
\end{theorem}

\section{Spectral approximation}
To analyze the spectral approximation, we introduce the solution operator $T:(L^2(\Omega))^2 \to (H^1_0(\Omega))^2$, which associates the solution $T\bff \in (H^1_0(\Omega))^2$ of the following source problem with every $\bff \in (L^2(\Omega))^2$:
\begin{equation*}
a(T\bff,\bv) =  (\bff,\bv), 	\quad \forall \bv \in (H^1_0(\Omega))^2.
\end{equation*}
The operator $T$ is well-defined because unique solvability for every $\bff \in (L^2(\Omega))^2$ is shown in \cite{A.Hansbo-P.Hansbo,Leguillon-Sanchez}. It is clear that the operator $T$ is bounded, self-adjoint and compact. In view of the definition of the solution operator $T$, if $(\omega^2, \bu) \in \mathbb{C}\setminus \{0\} \times (H^1_0(\Omega))^2$ is an eigenpair of (\ref{eq:variational_form}), then $(1/\omega^2, \bu)$ is an eigenpair for the operator $T$. In a similar way, we can define the corresponding discrete solution operator $T_h : (L^2(\Omega))^2 \to \bN_h(\Omega)$ by
\begin{equation*}
a_h(T_h\bff, \bv_h) = (\bff, \bv_h), \quad \forall \bv_h \in \what{\bN}_h(\Omega)
\end{equation*}
with $\bff \in (L^2(\Omega))^2$. Clearly, $T_h$ is also a bounded and self-adjoint operator. Notice that an eigenvalue $\xi_h$ of the operator $T_h$ is given by $\xi_h = 1/\omega_h^2$ where $\omega_h^2$ is an eigenvalue of the discrete problem (\ref{eq:disc-form}).

Before we show the uniform convergence of $T_h$ to $T$, we state some assumptions which are suggested to analyze the IFEM for the source problem associated with (\ref{eq:Model}) in \cite{Kwak-Jin-Kyeong}.
\renewcommand{\labelitemi}{$\bullet$}
\begin{itemize}
\item {(H1)}. There exists a constant $C>0$ such that
$$2\mu\|T\bff \|_{\wtilde{H}^2(\Omega)}+\lambda\|\Div T\bff\|_{\wtilde{H}^1(\Omega)} \leq C\|\mathbf{f}\|_{0,\Omega}.$$
\item {(H2)}. $\sigma_{ij}(T\bff) \in H^1(\Omega)$.
\end{itemize}
In fact, the hypothesis {(H1)} implies the regularity estimate which is known when the Lam\'{e} coefficients are continuous on the domain \cite{P.Hansbo_Lar2002}. On the other hand, such estimate for the interface problems is not available to the best of authors' knowledge. The hypothesis {(H2)} is required to analyze the consistency error of the scheme (\ref{a_h_form}). From now on, we assume the hypotheses {(H1)} and {(H2)}.

The following theorem \cite{Kwak-Jin-Kyeong} states the uniform convergence of $T_h$ to $T$ which plays an important role in spectral approximation. 
\begin{theorem} \label{thm:oper-conv}
There exists a constant $C > 0$ such that
\begin{equation*} 
\|T\bff- T_h\bff \|_{0,\Omega}+h\|T\bff-  T_h\bff\|_{a,h}\leq C h^2\|\bff\|_{0,\Omega}, \quad \forall \bff \in (L^2(\Omega))^2.
\end{equation*}
\end{theorem}

We are going to state the theoretical results of spectral approximation within the framework of \cite{Babuska-Osborn,Descloux-Nassif-Rappaz1978-1,Descloux-Nassif-Rappaz1978-2}. Most proofs of theorems stated below are analogous to \cite{Lee-Kwak-Sim} which deals with the IFEM for elliptic eigenvalue problems. Let us introduce some notations for theoretical results. To state the convergence of operators, we introduce an operator norm $\|L\|_{\mathscr{L}(X,Y)}$ for a bounded linear operator $L : X \to Y$ by
\begin{equation}
\|L\|_{\mathscr{L}(X,Y)} = \sup_{x \in X} \frac{\|Lx\|_Y}{\|x\|_X}. \label{eq:oper-norm}
\end{equation}
The distance between eigenspaces is evaluated by means of distance functions
$$ \mathrm{dist}_h (x, Y) \;= \inf_{y \in Y}\|x - y\|_{a,h}, \quad \mathrm{dist}_h(X,Y) \:= \sup_{x \in X, \|x\|_{a,h} = 1}\mathrm{dist}_h(x,Y), $$
where $X$ and $Y$ are closed subspaces of $\bH_h(\Omega)$. We denote by $\sigma(T)$ and $\rho(T)$ ($\sigma(T_h)$ and $\rho(T_h)$) the spectrum and resolvent set of the solution operator $T$ (resp. $T_h$), respectively. For any $z\in \rho(T)$, the resolvent operator $R_z(T)$ is defined by $R_z(T) = (z-T)^{-1}$ from $(L^2(\Omega))^2$ to $(L^2(\Omega))^2$ or from $(H^1_0(\Omega))^2$ to $(H^1_0(\Omega))^2$ and the discrete resolvent operator $R_z(T_h)$ is defined by $R_z(T_h) = (z-T_h)^{-1}$ from $\bH_h(\Omega)$ and $\bH_h(\Omega)$ \cite{Kato}.

To show that the resolvent operators $R_z(T)$ and $R_z(T_h)$ are well-defined and bounded, we introduce the following theorem.
\begin{theorem}
For $z\in \rho(T)$, $z\neq 0$ and $h$ small enough, there are constants $C_1,\,C_2>0$ depending on only $\Omega$ and $|z|$ such that
\begin{equation} \label{eq:resolBdd1}
\|(z-T)\bff\|_{a,h} \geq C_1\|\bff\|_{a,h}, \quad \forall \bff \in \bH_h(\Omega),
\end{equation}
\begin{equation} \label{eq:resolBdd2}
\|(z-T_h)\bff\|_{a,h} \geq C_2\|\bff\|_{a,h}, \quad \forall \bff\in \bH_h(\Omega).
\end{equation}
\label{thm:resolBdd}
\end{theorem}
\begin{proof}
The proof of the first inequality (\ref{eq:resolBdd1}) is essentially identical to that of Lemma 4.1 from \cite{Lee-Kwak-Sim}. The second inequality (\ref{eq:resolBdd2}) follows from the first inequality (\ref{eq:resolBdd1}) and Theorem \ref{thm:oper-conv}, (see Lemma 1 in \cite{Descloux-Nassif-Rappaz1978-1}).
\end{proof}

Let $\xi$ be an eigenvalue of $T$ with algebraic multiplicity $n$ and $\Lambda$ be a Jordan curve in $\mathbb{C}$ containing $\xi$, which lies in $\rho(T)$ and does not enclose any other points of $\sigma(T)$. We define the spectral projection $E(\xi)$ from $(L^2(\Omega))^2$ into $(H^1_0(\Omega))^2$ by
$$E(\xi) = \frac{1}{2\pi i}\int_{\Lambda} R_z(T)\, dz. $$
Owing to Theorem \ref{thm:resolBdd}, we can define the discrete spectral projection $E_h(\xi)$ from $(L^2(\Omega))^2$ into $\bH_h(\Omega)$ for $h$ small enough by
$$E_h(\xi) = \frac{1}{2\pi i} \int_{\Lambda} R_z(T_h)\, dz. $$
We simply denote the projections $E(\xi)$ and $E_h(\xi)$ by $E$ and $E_h$, respectively.

\begin{theorem}
The discrete projection operator $E_h$ converges uniformly to the projection operator $E$, i.e., it holds that
\begin{equation}
\lim_{h \to 0}\|E - E_h\|_{\mathscr{L}((L^2(\Omega))^2, \bH_h(\Omega))} = 0. \nonumber
\end{equation}
\end{theorem}
\begin{proof}
We remark the residual identity
$$R_z(T) - R_z(T_h) = R_z(T_h)(T-T_h)R_z(T),$$
so that
\begin{align*}
\|E - E_h\|_{\mathscr{L}((L^2(\Omega))^2, \bH_h(\Omega))} \leq & \| R_z(T_h)  \|_{\mathscr{L}(\bH_h(\Omega), \bH_h(\Omega))}\|T-T_h\|_{\mathscr{L}((L^2(\Omega))^2, \bH_h(\Omega))} \\
	&\cdot \|R_z(T)\|_{\mathscr{L}((L^2(\Omega))^2, (L^2(\Omega))^2)}.
\end{align*}
By Theorem \ref{thm:resolBdd} and Fredholm alternative \cite{Kato}, the resolvent operators $R_z(T_h)$ and $R_z(T)$ are bounded for $h$ small enough. In addition, the operator $T_h$ converges to $T$ uniformly by Theorem \ref{thm:oper-conv}. Therefore, we conclude the proof.
\end{proof}

Finally, we can say that the discrete problem (\ref{eq:disc-form}) is a spectrally correct approximation of the problem (\ref{eq:variational_form}), provided that the following theorem holds \cite{Descloux-Nassif-Rappaz1978-1}. For the proofs of following result, we refer to those of Theorem 1,2,3 and 6 from \cite{Descloux-Nassif-Rappaz1978-1}.
\begin{theorem} \label{thm:spect-correct}

$\bullet$ (Non-pollution of the spectrum) Let $A \subset \mathbb{C}$ be an open set containing $\sigma(T)$. Then for sufficiently small $h$, $\sigma(T_h) \subset A$. \\
$\bullet$ (Non-pollution of the eigenspace) $$\lim_{h \to 0} \mathrm{dist}_h(E_h(\bH_h(\Omega)), E((H^1_0(\Omega))^2))=0.$$
$\bullet$ (Completeness of the eigenspace) $$\lim_{h \to 0} \mathrm{dist}_h(E((H^1_0(\Omega))^2), E_h(\bH_h(\Omega)))=0.$$
$\bullet$ (Completeness of the spectrum) For all $z \in \sigma(T)$,
$$\lim_{h \to 0} \mathrm{dist}_h (z, \sigma(T_h)) = 0.$$

\end{theorem}
It remains to show the convergence analysis of eigenvalues. The convergence rate of eigenvalues is obtained by the spectral properties of compact operators and the uniform convergence of the operator $T_h$ to $T$ in Theorem \ref{thm:oper-conv}.
\begin{theorem} \label{thm:eig-conv}
Let $\xi$ be an eigenvalue of $T$ with multiplicity $n$. Then for $h$ small enough there exist $n$ eigenvalues $\{\xi_{1,h}, ... , \xi_{n,h}\}$ of $T_h$ which converge to $\xi$ as follows
\begin{equation*}
\sup_{1\leq i \leq n}|\xi - \xi_{i,h}| \leq C h^2,
\end{equation*}
where a positive constant $C$ is independent of $\xi$ and $h$.
\end{theorem}

\begin{proof}
The existence of $\xi_{i,h}$ is a direct consequence of Theorem \ref{thm:spect-correct}. To estimate the convergence rate of $\xi_{i,h}$, we introduce some auxiliary operators. Let $\Phi_h$ and $\wtilde T$ be the restriction of operators $E_h$ and $T$ to $E((L^2(\Omega))^2)$, respectively. Following the arguments in \cite{Babuska-Osborn, Osborn}, we have that the inverse $\Phi_h^{-1} : E_h(\bH_h(\Omega)) \to E((L^2(\Omega))^2)$ is bounded for $h$ small enough. Hence we can define $\wtilde T_h := \Phi^{-1}_h T_h \Phi_h$ and $S_h := \Phi_h^{-1}E_h$. Note that the operator $S_h$ is bounded and $S_h \bff = \bff$ for any $\bff \in E((L^2(\Omega))^2)$. The auxiliary operators $\wtilde T$, $\wtilde T_h$, $S_h$ and $\Phi_h$ provide a following property, for any $\bff \in E((L^2(\Omega))^2)$,
\begin{align}
(\wtilde{T}-\wtilde{T}_h)\bff &= T\bff - \Phi_h^{-1}T_h \Phi_h \bff \nonumber \\
	&= S_h T \bff - \Phi_h^{-1} T_h E_h \bff \nonumber \\
	&= S_h T \bff - \Phi_h^{-1}E_hT_h \bff \nonumber  \\ 
	&= S_h(T - T_h)\bff. \nonumber
\end{align}
In view of definition of operator norm (\ref{eq:oper-norm}) and Theorem \ref{thm:oper-conv}, we have
\begin{align*}
\sup_{1\leq i \leq n} |\xi - \xi_{i,h}| &\leq C \| \wtilde T - \wtilde T_h\|_{\mathscr{L}(E((L^2(\Omega))^2), E((L^2(\Omega))^2))} \\
	&= C \sup_{\bff \in E((L^2(\Omega))^2)} \frac{\|(\wtilde T - \wtilde T_h)\bff\|_{0,\Omega}}{\|\bff\|_{0,\Omega}} \\
	&= C \sup_{\bff \in E((L^2(\Omega))^2)} \frac{\|S_h(T - T_h) \bff\|_{0,\Omega}}{\|\bff \|_{0,\Omega}} \\
	&\leq C \sup_{\bff \in E((L^2(\Omega))^2)} \frac{\|(T - T_h) \bff\|_{0, \Omega}}{\|\bff \|_{0,\Omega}} \\
	&\leq C h^2\;.
\end{align*}

\end{proof}
\begin{remark}
Overall, we show that our IFEM is spurious-free and has optimal convergence property by Theorem \ref{thm:spect-correct} and Theorem \ref{thm:eig-conv}. Although uniform convergence of solution operator which is essential basis for spectral analysis is based on the hypothesis (H1) and (H2), a variety of numerical results reported in the next section corroborate our theoretical results.
\end{remark}

\section{Numerical results}
In this section we present a series of numerical experiments to verify the theoretical analysis for the approximation of the eigenvalue problem (\ref{eq:Model}) in the previous sections. We recall the definition of the Lam\'{e} coefficients of a material
$$\lambda = \frac{E \nu}{(1+\nu)(1-2\nu)},\,\,\; \mu = \frac{E}{2(1+\nu)},$$
where $E$ is the Young's modulus and $\nu$ is the Poisson ratio. We carry out numerical tests for the cases of the compressible elastic materials ($\nu < 0.5$) and the nearly incompressible elastic materials ($\nu \approx 0.5$) with various shapes of interface in Figure \ref{fig:interface}. For a square domain $\Omega = [-1,1]^2$, we use uniform triangle meshes with mesh size $h = 2/N$ where the refinement parameter $N$ is the number of elements on each edge. Since analytical expressions for the eigenvalues are not available for all of the examples, we use the numerical results on a sufficiently refined mesh as the reference eigenvalues in order to estimate the order of convergence. In all the numerical examples, the IFEM is implemented in a C++ code and the eigenvalues are computed with ARPACK \cite{Lehoucq-Sorensen-Yang}.

\begin{figure}[hbt]
\begin{center}
\begin{pspicture}(-0,-2.2)(7,7.0)\footnotesize
\psset{xunit=1.65, yunit=1.65}
\rput(0.5,3){
\pspolygon(-1,-1)(1,-1)(1,1)(-1,1)
\rput(-0.62,0.6){$\Omega^+$} \rput(-0.12,0.12){$\Omega^-$}
\psplot{-0.6}{0.6}{ 0.36 x 2 exp sub sqrt }
\psplot{-0.6}{0.6}{ 0.36 x 2 exp sub sqrt -1. mul}
\rput(0.1,-1.4){Circular interface}
}

\rput(4.0,3){
\pspolygon(1,-1)(1,1)(-1,1)(-1,-1)(1,-1)

\rput(-0.62,0.4){$\Omega^+$} \rput(-0.12,0.12){$\Omega^-$}
\psplot{-0.6}{0.6}{ 1 x .6 div 2 exp sub sqrt .3 mul}
\psplot{-0.6}{0.6}{ 1 x .6 div 2 exp sub sqrt -.3 mul}
\rput(0.1,-1.4){Elliptical interface}
 }

\rput(0.5,0){
\pspolygon(-1,-1)(1,-1)(1,1)(-1,1)
\psplot{-1}{1}{x 0.5 mul 0.2 sub}
 \rput(0.3,-0.5){$\Omega^-$}  \rput(-0.1,0.7){$\Omega^+$}
\rput(0.,-1.4){Straight-line interface}
}

\rput(4.0,0){
\pspolygon(-1,-1)(1,-1)(1,1)(-1,1)
\pscircle(0,0){0.429} 
\pscircle(0.5,0.5){0.3135} 
\pscircle(-0.5,0.5){0.3135} 
\pscircle(-0.5,-0.5){0.3135} 
\pscircle(0.5,-0.5){0.3135} 
 \rput(0.5,-0.5){$\Omega^-$}  \rput(0.5,0.5){$\Omega^-$}
 \rput(-0.5,-0.5){$\Omega^-$} \rput(-0.5,0.5){$\Omega^-$}
 \rput(0.,0.61){$\Omega^+$} \rput(0.,0.05){$\Omega^-$}
\rput(0.,-1.4){Multiple interfaces}
}
\end{pspicture}
\caption{Domain and interfaces in Examples 1,2,3,4 and 5}\label{fig:interface}
\end{center}
\end{figure}
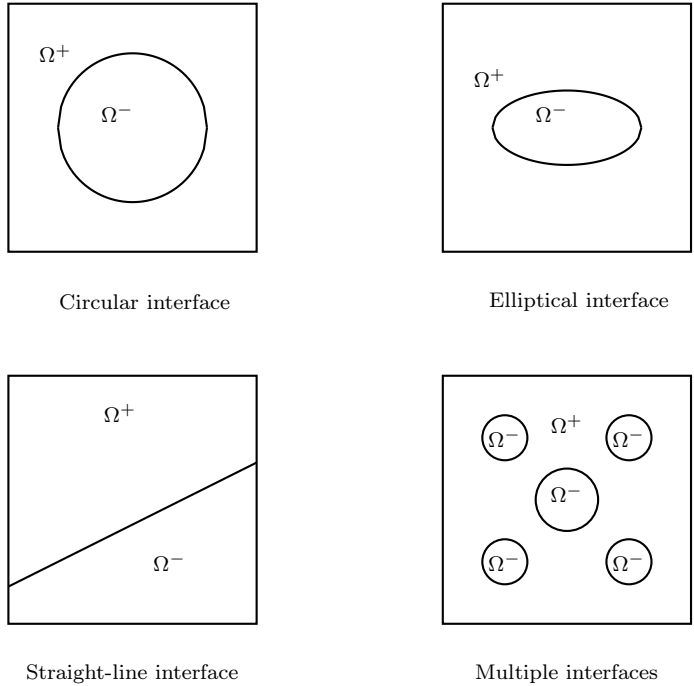

\textbf{Example 1} ({\it Circular interface}). In this example, we consider the eigenvalue problem (\ref{eq:Model}) with a circular interface. The interface $\Gamma$ is a circle with radius $r=0.6$ dividing $\Omega= [-1,1]^2$ into subdomains $\Omega^+$ and $\Omega^-$ as follows,
\begin{equation}
\Omega^+ := \{(x,y) : x^2 + y^2 > r^2\}, \quad \Omega^- := \{(x,y) : x^2 + y^2 < r^2 \}.
\end{equation}
We set the Lam\'{e} coefficients as $(\mu^-,\mu^+) = (0.5,5),(5,0.5),\, \lambda^{\pm} = 5\mu^{\pm}$ and Poisson ratio $\nu^{\pm} \approx 0.417$.
Table \ref{table:comp-circle} shows the first six eigenvalues and their rates of convergence. The first columns are the reference eigenvalues computed with very fine mesh size $h = 2^{-9}$ and the other columns are the eigenvalues obtained with IFEM for varying $h$. We observe that the convergence rates of the eigenvalue errors are quadratic. An eigenfunction for eigenvalue $\omega_4^2$, together with each $x$ and $y$-component of eigenfunction, are depicted in Figure \ref{fig:circle-eigvec}.

\renewcommand{\arraystretch}{1.4}
\begin{table}[!ht] \footnotesize
\caption{First six eigenvalues computed by IFEM with circular interface for compressible materials. The reference eigenvalues $\omega^2_{ref}$ in the first column are computed with $h = 1/2^{9}$. The numbers in parentheses show convergence rates.}
\centering
\begin{tabular}{c|crrrr}
\hline
\multicolumn{6}{l}{Circular interface - $(\mu^-,\mu^+) = (0.5,5),\, \lambda^{\pm} = 5\mu^{\pm}$} \\ \hline
$\omega^2_{ref}$	  & \multicolumn{1}{c}{$h=1/2^{3}$} & \multicolumn{1}{c}{$h=1/2^{4}$(ord)}  & \multicolumn{1}{c}{$h=1/2^{5}$(ord)}  & \multicolumn{1}{c}{$h=1/2^{6}$(ord)}  & \multicolumn{1}{c}{$h=1/2^{7}$(ord)} \\ \hline
18.824 &	20.409		&	19.197	(2.09) &	18.915	(2.03) &	18.846	(2.02) &	18.829	(2.07) 	 \\
23.384 &	24.040		&	23.545	(2.03) &	23.420	(2.16) &	23.392	(2.19) &	23.385	(2.51) 	 \\
23.385 &	24.953		&	23.832 (1.81) &	23.499	(1.97) &	23.413	(2.02) &	23.391	(2.12) 	 \\
40.666 &   43.609		&	41.349	(2.10) &	40.831	(2.05) &	40.706	(2.03) &	40.675	(2.09) 	 \\
40.667 &   49.381		&	42.751	(2.06) &	41.182	(2.01) &	40.795 (2.01) &	40.698	(2.07) 	 \\
45.938 &   51.811		&	47.431	(1.98) &	46.291	(2.08) &	46.023	(2.05) &	45.957 (2.14) 	\\

\end{tabular}

\begin{tabular}{c|crrrr}
\hline
\multicolumn{6}{l}{Circular interface - $(\mu^-, \mu^+) = (5,0.5),\, \lambda^{\pm} = 5\mu^{\pm}$} \\ \hline
$\omega^2_{ref}$	  & \multicolumn{1}{c}{$h=1/2^{3}$} & \multicolumn{1}{c}{$h=1/2^{4}$(ord)}  & \multicolumn{1}{c}{$h=1/2^{5}$(ord)}  & \multicolumn{1}{c}{$h=1/2^{6}$(ord)}  & \multicolumn{1}{c}{$h=1/2^{7}$(ord)} \\ \hline
7.151 &		7.356 	&	7.205	(1.93) &	7.1652 (2.00) &	7.155	(2.00) &	7.1524	(2.06) 	 \\
10.121 &	10.165		&	10.135	(1.63) &	10.124	(2.21) &	10.121	(1.99) &	10.121 	(2.20) 	 \\
10.121 &	10.313		&	10.177 (1.76) &	10.135	(1.98) &	10.124	(1.98) &	10.121	(2.07) 	 \\
24.205 &   25.696		&	24.585	(1.97) &	24.292	(2.12) &	24.224	(2.15) &	24.208	(2.35) 	 \\
24.205 &   26.708		&	24.863	(1.93) &	24.368	(2.01) &	24.245 (2.03) &	24.214	(2.12) 	 \\
32.257 &   34.569		&	32.937	(1.77) &	32.439	(1.90) &	32.303	(1.97) &	32.268 (2.06) 	\\ \hline
\end{tabular}
\label{table:comp-circle}
\end{table}

\begin{figure}[!ht]
	\centering
	 \includegraphics[width=0.44\textwidth, height=0.44\textwidth]{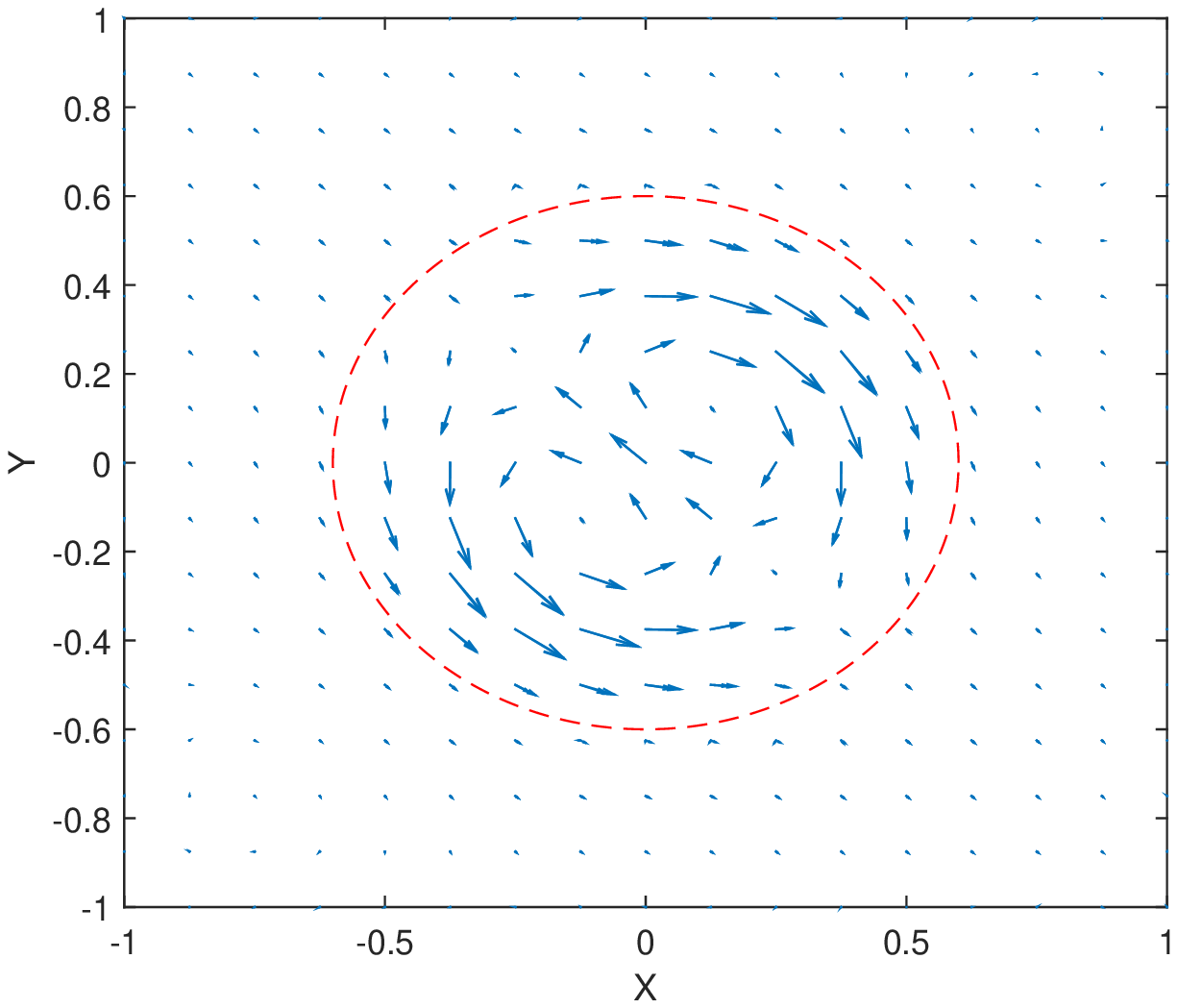} \\
	 \includegraphics[width=0.47\textwidth, height=0.43\textwidth]{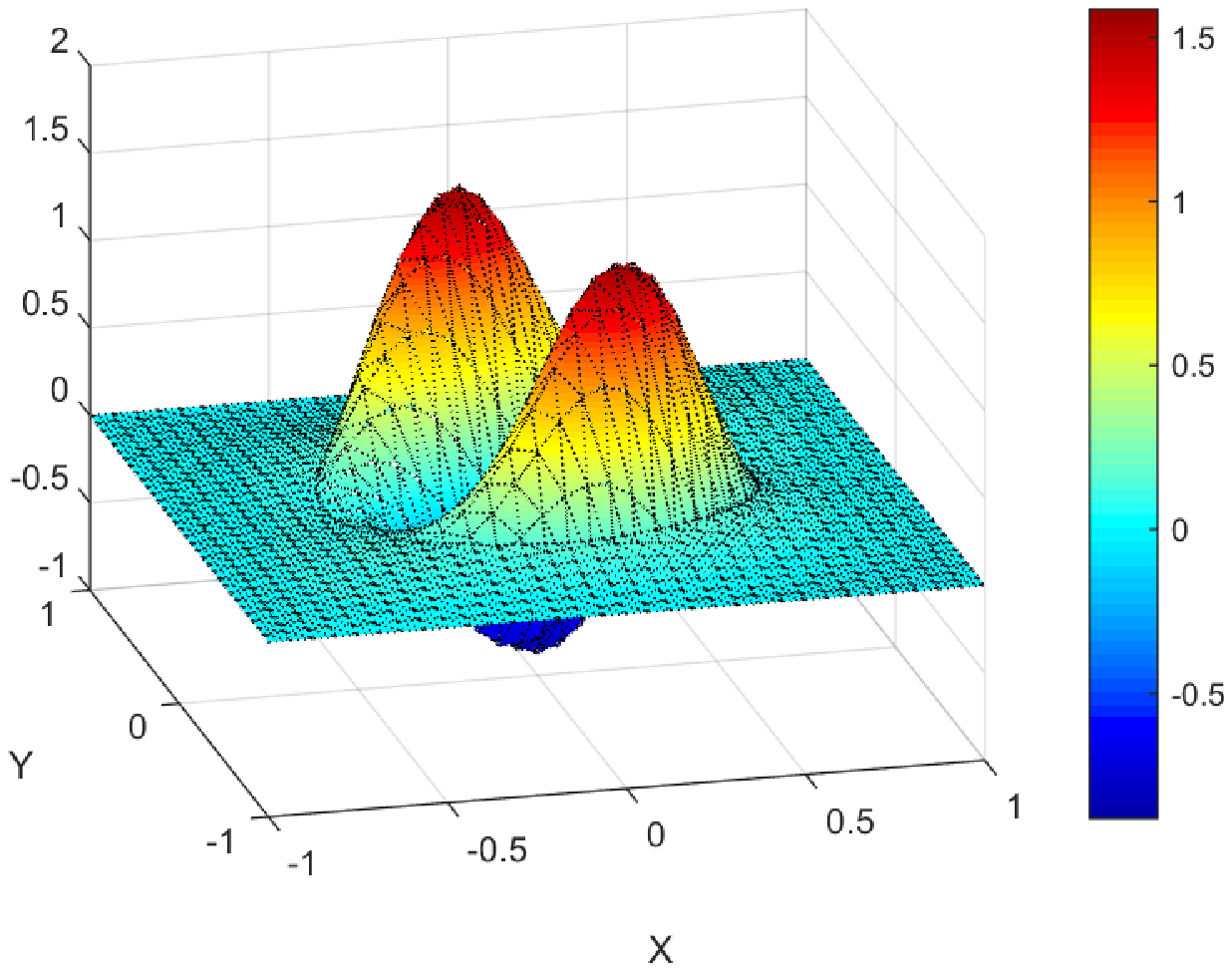}
	\includegraphics[width=0.47\textwidth, height=0.43\textwidth]{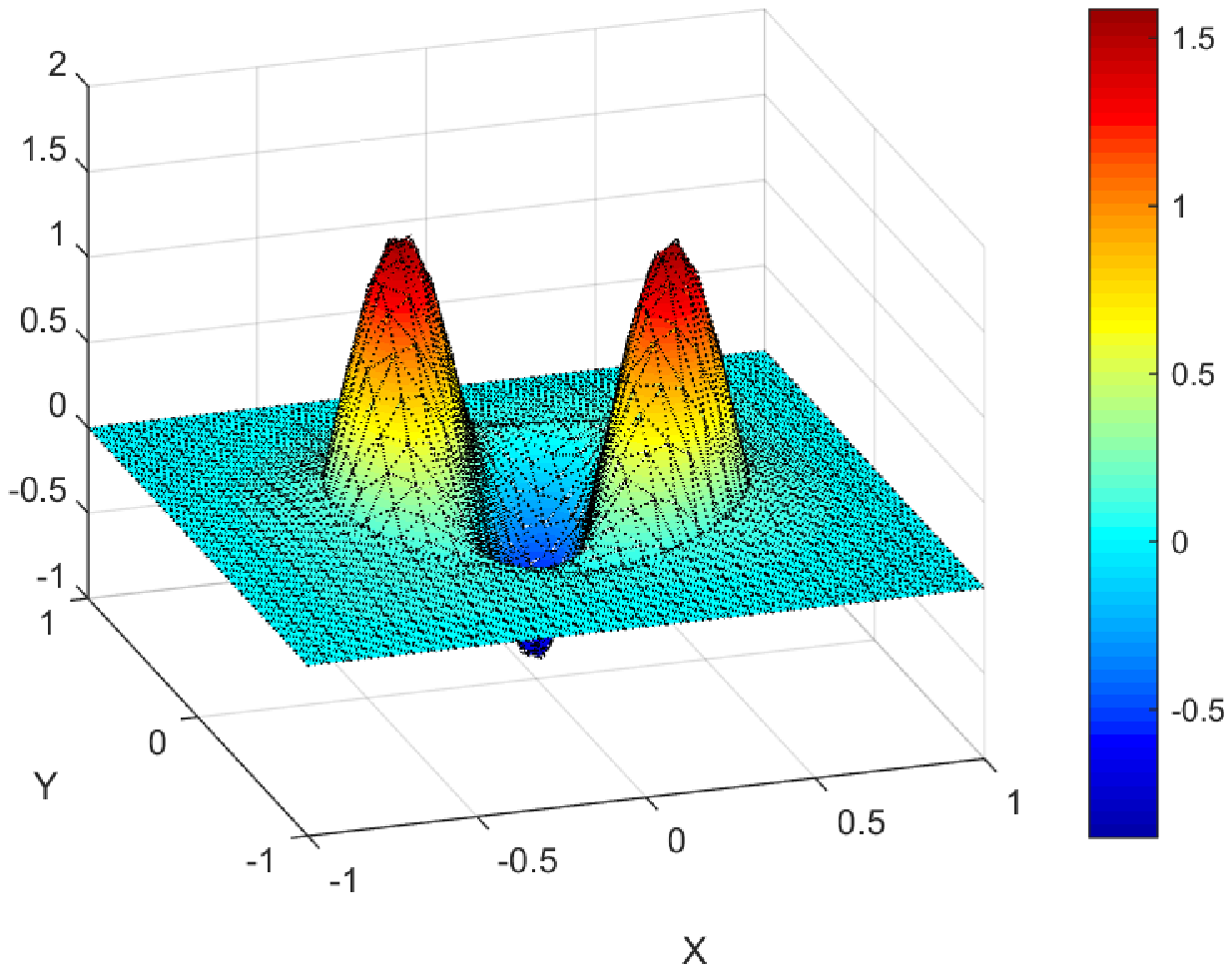}
	\caption{The figure above is an eigenfunction of $\omega^2_4$ when $(\mu^-, \mu^+)=(0.5,5),\,\lambda^{\pm} = 5\mu^{\pm}$ in Example 1. The figures below are x-component and y-component of eigenfunction of $\omega^2_4$.}
	\label{fig:circle-eigvec}
\end{figure}

\textbf{Example 2} ({\it Elliptical interface}). The second example concerns an elliptical interface given by $\Gamma = \{(x,y): x^2/a^2 + y^2/b^2 = 1\}$ where $a=0.6$ and $b=0.3$. We set subdomain $\Omega^-$ to be an interior and $\Omega^+$ to be the other part of the domain $\Omega$, i.e.,
\begin{equation}
\Omega^+ := \{(x,y) : x^2/a^2 + y^2/b^2 > 1\}, \quad \Omega^- := \{(x,y) : x^2/a^2 + y^2/b^2 < 1 \}.
\end{equation}
Let Lam\'{e} coefficients be $(\mu^-, \mu^+) = (0.5,5),(5,0.5)$ and $\lambda^{\pm} = 5\mu^{\pm}$. In Figure \ref{fig:ellipse-comp}, we show the errors of the first four eigenvalues by IFEM and corresponding order of convergence. The reference solution is the numerical results on a refined mesh with mesh size $h = 2^{-9}$. Even though the interface becomes shaper than the circular interface, the optimal convergence for eigenvalues is obtained. We display an eigenfunction of $\omega^2_3$ in Figure \ref{fig:ellipse-eigvec} which is analogous to Figure \ref{fig:circle-eigvec}.

\begin{figure}[!ht]
	\centering
	 \includegraphics[width=0.45\textwidth, height=0.45\textwidth]{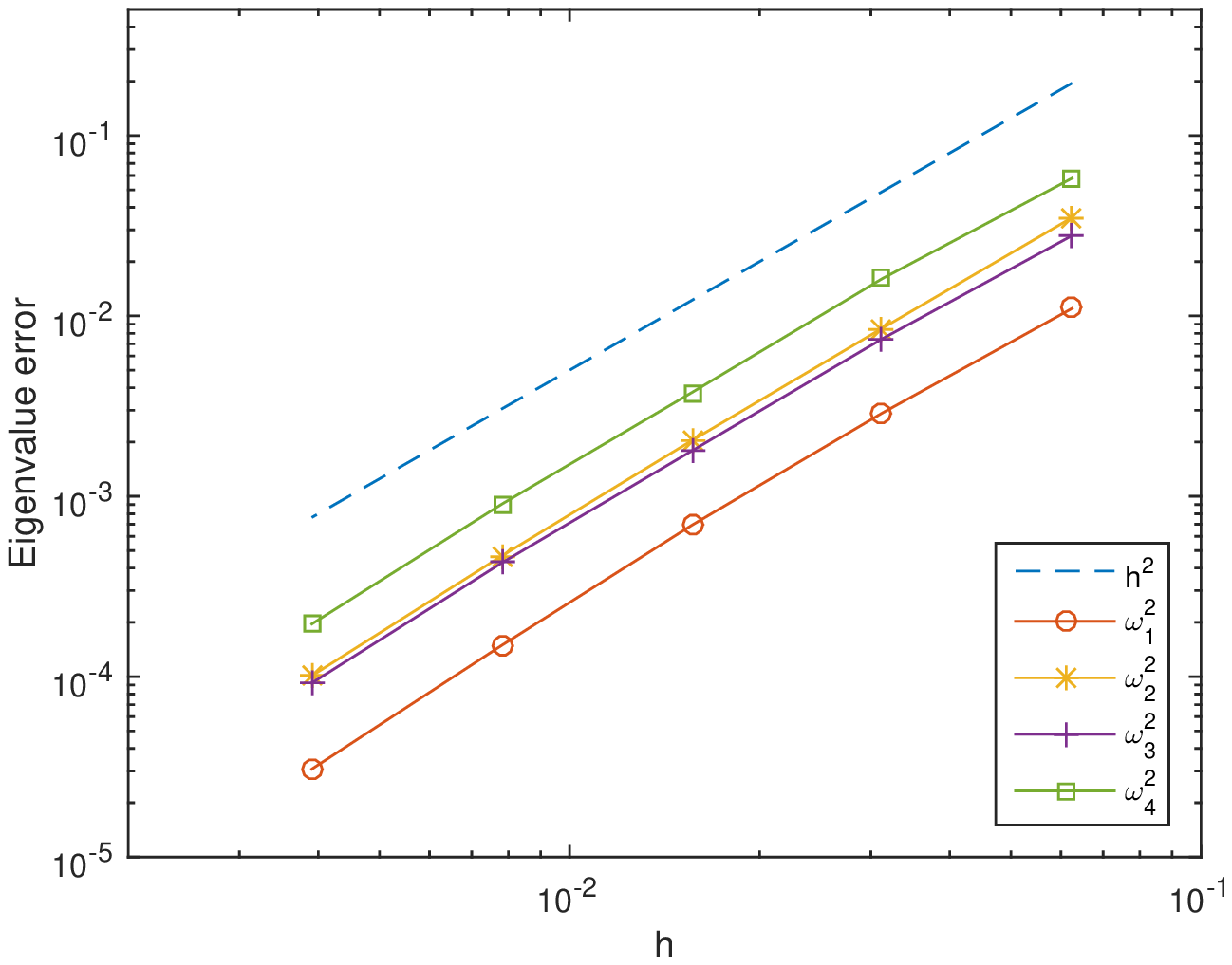}
	\includegraphics[width=0.45\textwidth, height=0.45\textwidth]{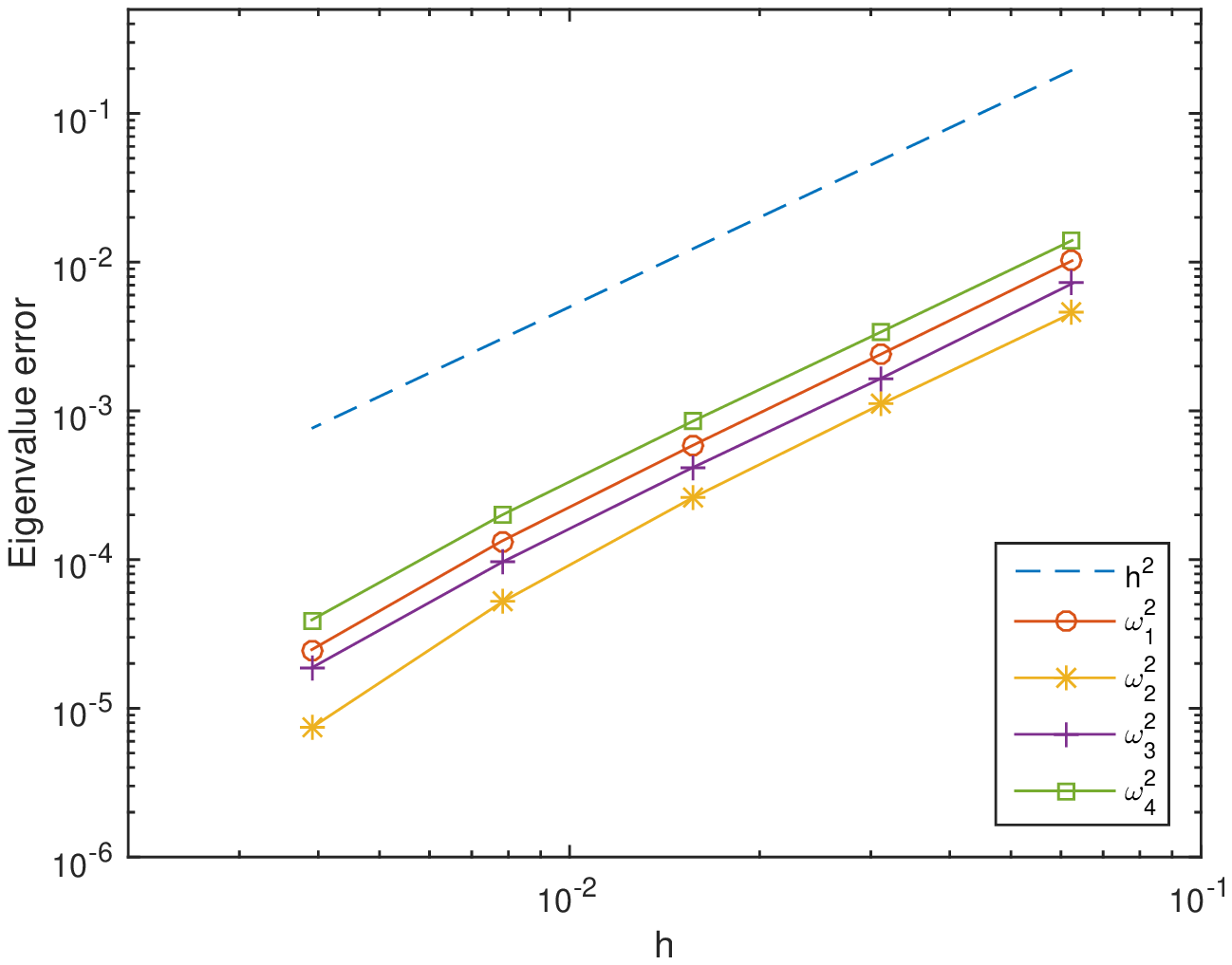}
	\caption{The log-log plots of $h$ versus the relative error of the first four eigenvalues with an elliptical interface for the case of $(\mu^-, \mu^+) = (0.5,5)$ (left) and $(\mu^-, \mu^+) = (5,0.5)$ (right) in Example 2. The broken line represents the optimal convergence rate.}
	\label{fig:ellipse-comp}
\end{figure}

\begin{figure}[!ht]
	\centering
	 \includegraphics[width=0.44\textwidth, height=0.44\textwidth]{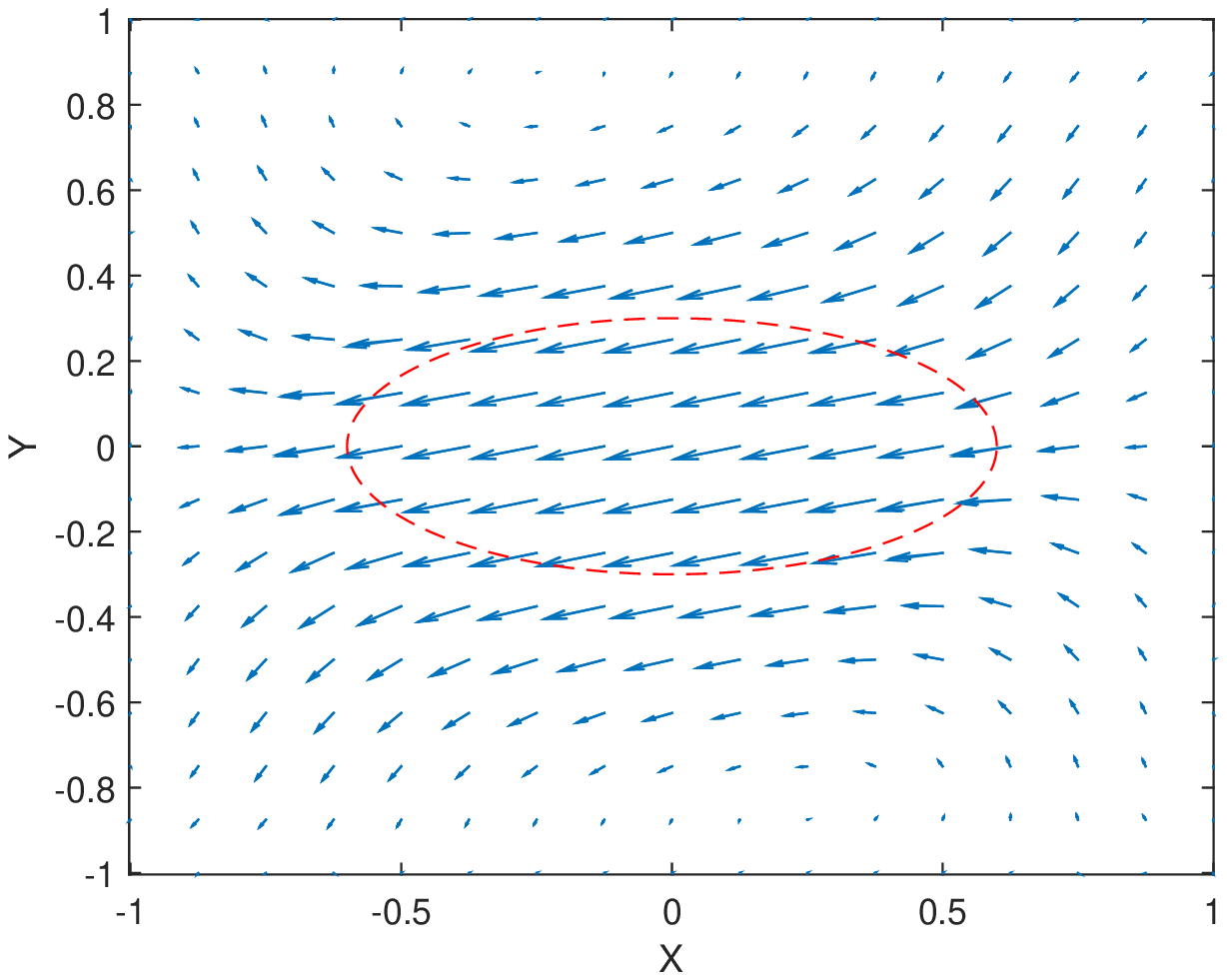}\\
	 \includegraphics[width=0.47\textwidth, height=0.43\textwidth]{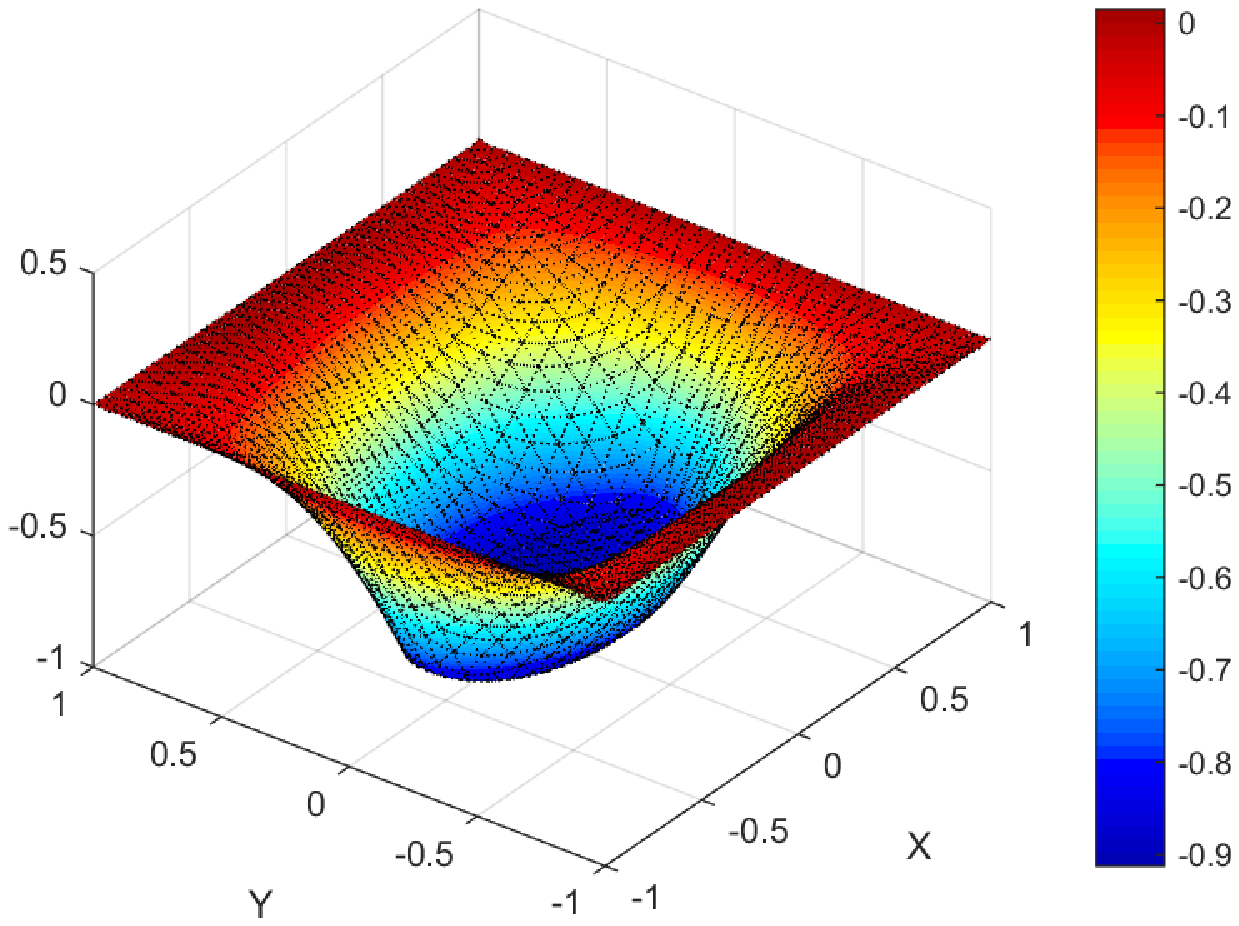}
	\includegraphics[width=0.47\textwidth, height=0.43\textwidth]{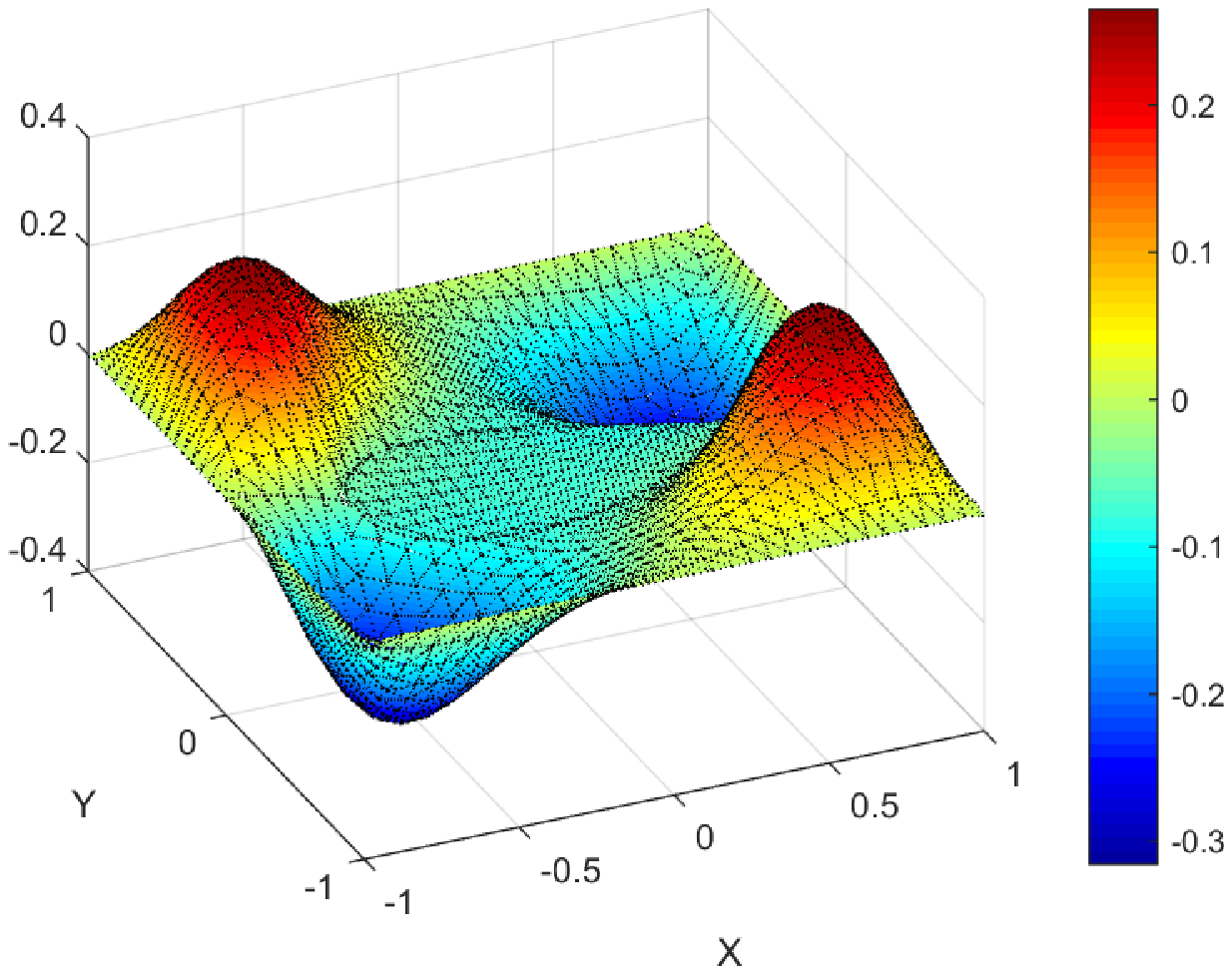}
	\caption{An eigenfunction of $\omega^2_3$ for the case of an elliptical interface when $\mu^- = 5,\, \mu^+=0.5,\lambda^{\pm} = 5\mu^{\pm}$ in Example 2. The figures below are x-component of eigenfunction on the  left and y-component of eigenfunction on the right.}
	\label{fig:ellipse-eigvec}
\end{figure}

\textbf{Example3} ({\it Straight-line interface}). We let an interface be a straight line as $\Gamma = \{(x,y) : y = 0.5x - 0.2\}$ and the subdomians $\Omega^+$ and $\Omega^-$ be
\begin{equation}
\Omega^+ := \{(x,y) : y > 0.5x - 0.2\}, \quad \Omega^- := \{(x,y) : y < 0.5x -0.2\}.
\end{equation}
Lam\'{e} coefficients are the same as previous examples, $(\mu^-,\mu^+) = (0.5,5),\,(5,0.5)$ and $\lambda^{\pm} = 5\mu^{\pm}$. In Figure \ref{fig:line-comp}, we show the errors of the first four eigenvalues computed with IFEM. This figure also presents that the rates of convergence are quadratic. Note that in this example the interface meets the boundary of the domain. Nevertheless, the order of convergence of the eigenvalue is optimal. Figure \ref{fig:line-eigvec} shows the computed eigenfunction corresponding to $\omega_4^2$.

\begin{figure}[!ht]
	\centering
	 \includegraphics[width=0.45\textwidth, height=0.45\textwidth]{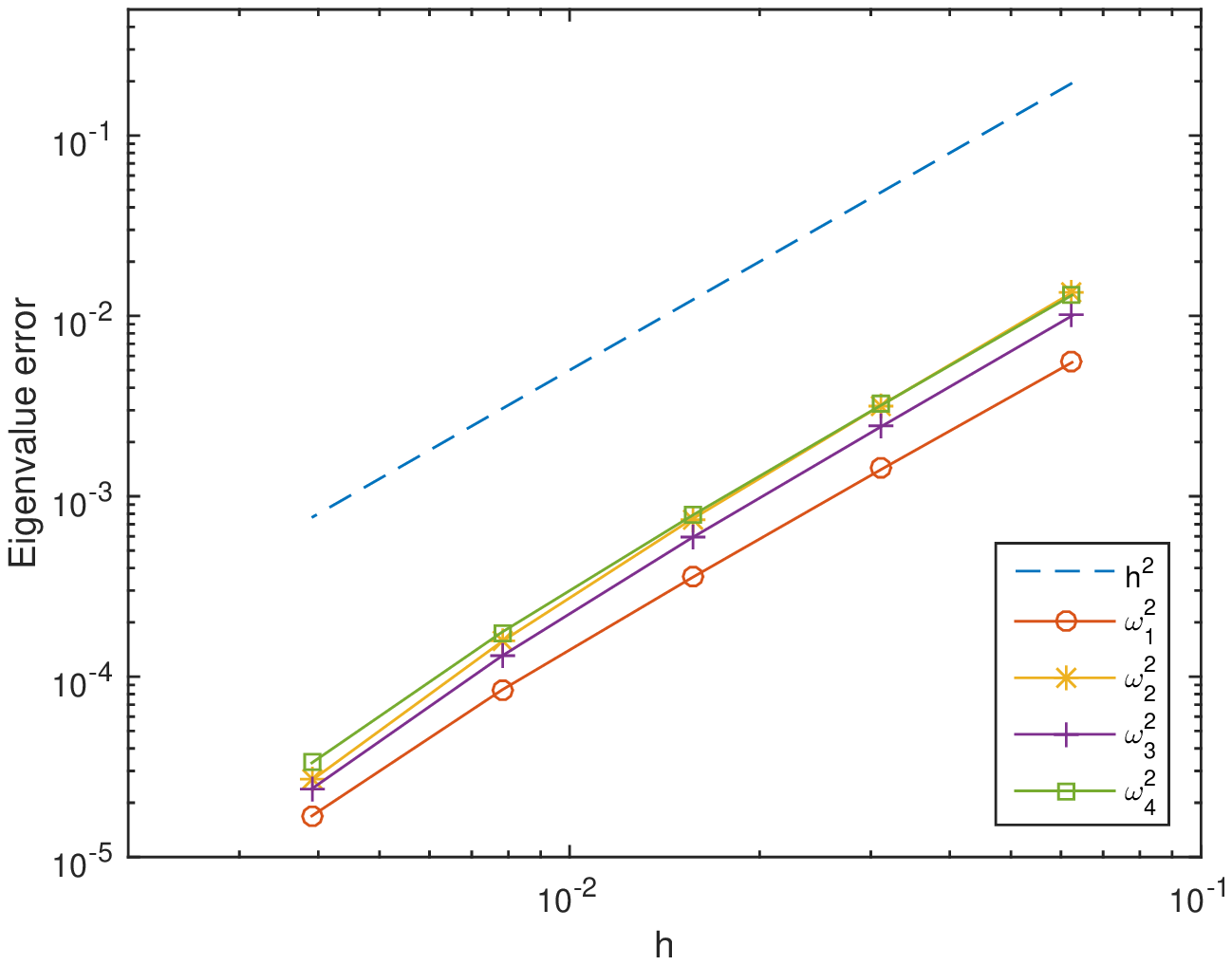}
	\includegraphics[width=0.45\textwidth, height=0.45\textwidth]{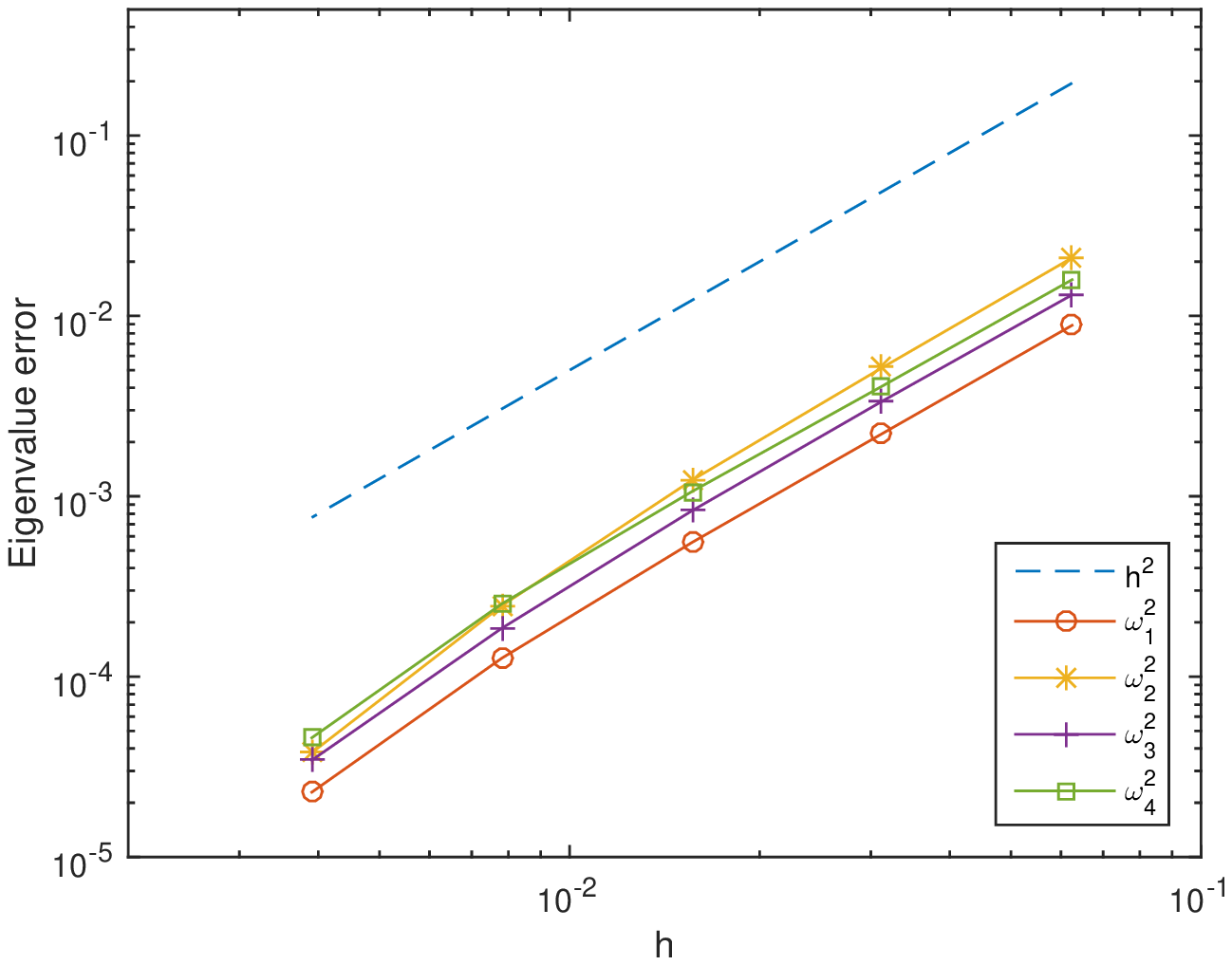}
	\caption{The log-log plots of $h$ versus the relative error of the first four eigenvalues with a straight-line interface for the case of $(\mu^-, \mu^+) = (0.5,5)$ (left) and $(\mu^-, \mu^+) = (5,0.5)$ (right) in Example 3. The broken line represents the optimal convergence rate.}
	\label{fig:line-comp}
\end{figure}

\begin{figure}[!ht]
	\centering
	 \includegraphics[width=0.44\textwidth, height=0.44\textwidth]{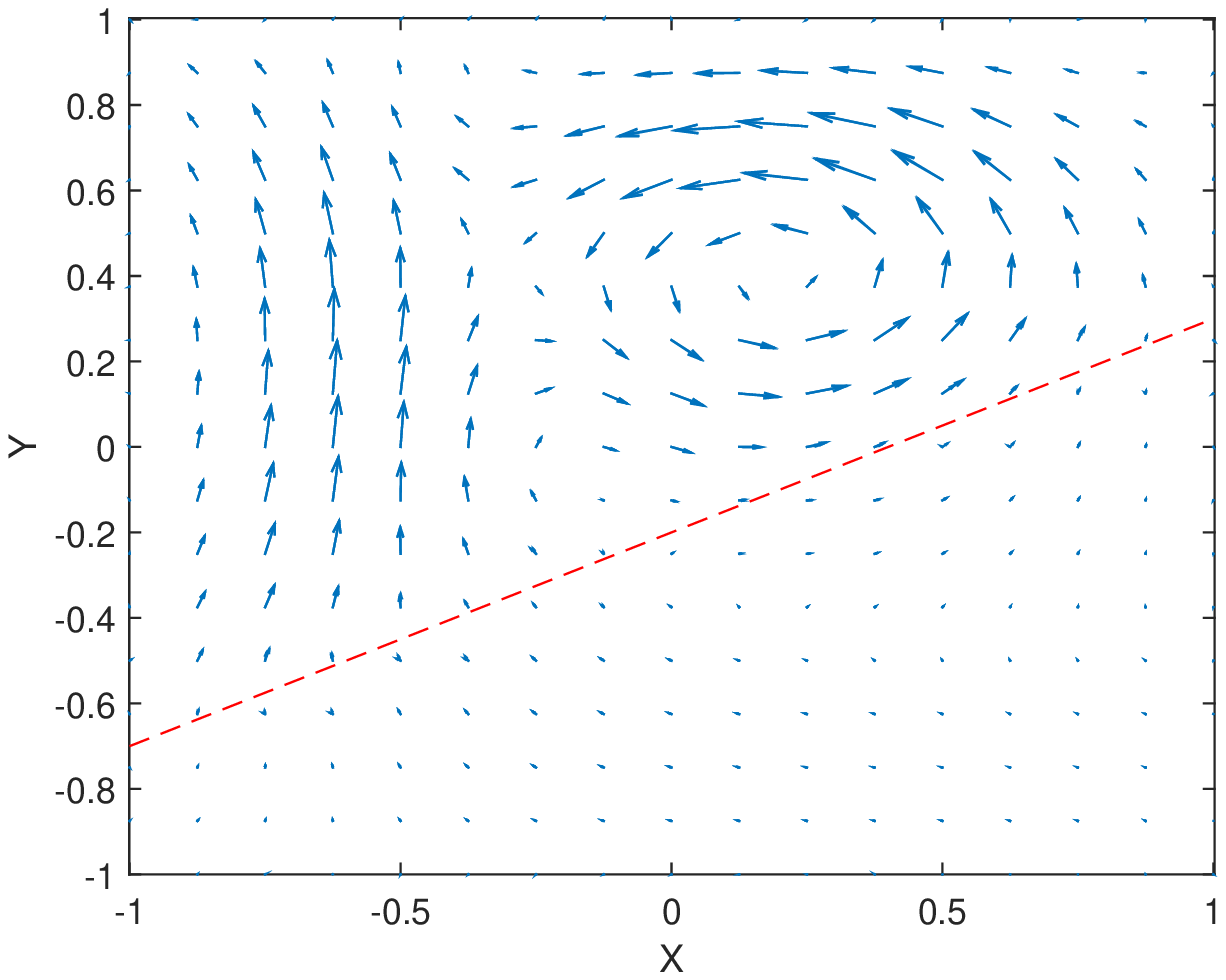} \\
	 \includegraphics[width=0.47\textwidth, height=0.43\textwidth]{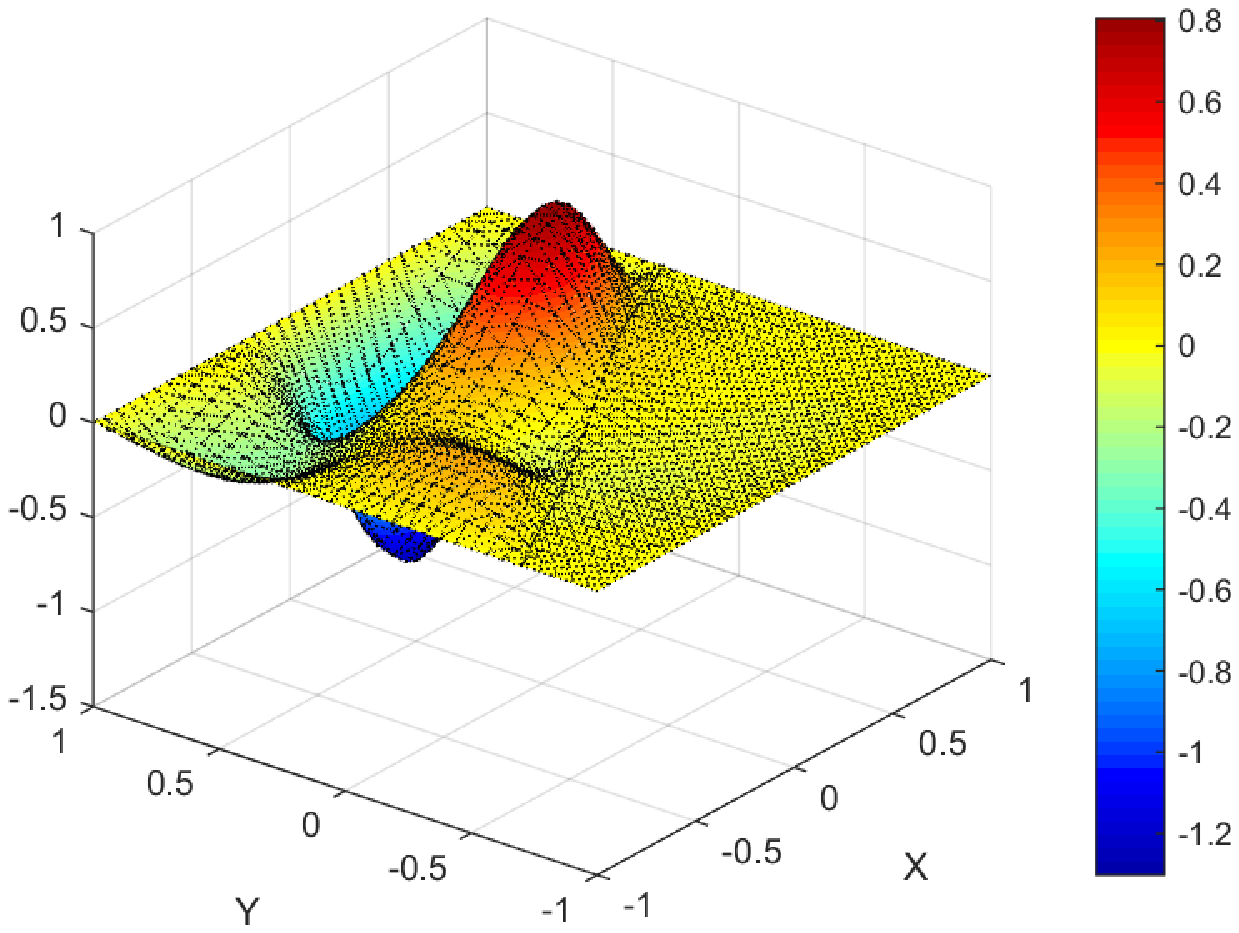}
	\includegraphics[width=0.47\textwidth, height=0.43\textwidth]{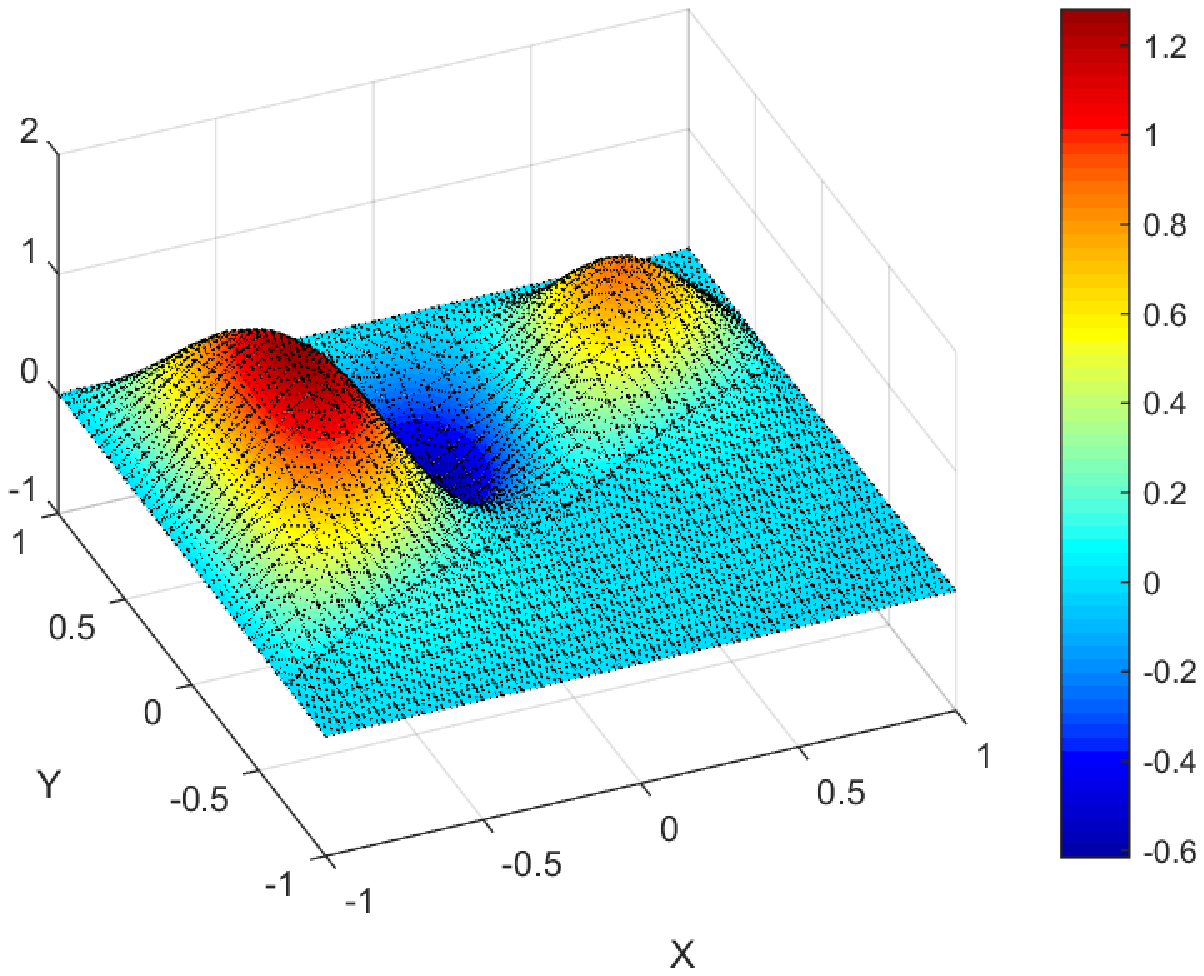}
	\caption{Eigenfunction of $\omega^2_4$ when $(\mu^-, \mu^+)=(0.5,5),\,\lambda^{\pm} = 5\mu^{\pm}$ in Example 3 (above), x-component of eigenfunction (below on the  left), and y-component of eigenfunction (below on the right).}
	\label{fig:line-eigvec}
\end{figure}

\textbf{Example4} ({\it Multiple interfaces}) In this case, we solved the problem (\ref{eq:Model}) with 5 circular interfaces. Let subdomains $\Omega^-$ and $\Omega^+$ be as follows
\begin{align*}
\Omega^- &= \cup_{i=1}^{5} \{(x,y) : (x-a_i)^2 + (y-b_i)^2 < r_i\}, \\
\Omega^+ &= \Omega \setminus \Omega^-,
\end{align*}
where $(a_1,b_1) = (0,0),\, r_1 = 0.26$ and $(a_i,b_i) = (\pm 0.5,\pm 0.5),\, r_i = 0.19$ for $ 2 \leq i \leq 5$ (see Figure\ref{fig:interface}). The Lam\'{e} coefficients are chosen as follows
: $(\mu^-,\lambda^-,\nu^-) = (1,2,0.33)$, $(\mu^{+},\lambda^{+},\nu^+) = (30,36,0.27)$. Figure \ref{fig:multiple-err} illustrates the error and the rates of convergence for $\omega^2_i,\,(1\leq i\leq 4)$ by IFEM. The results in Figure \ref{fig:multiple-err} are in good agreement with our theoretical analysis in the previous section. Figure \ref{fig:multiple-eigvec} depicts an eigenfunction of $\omega^2_7$.

\begin{figure}[!ht]
	\centering
	 \includegraphics[width=0.5\textwidth, height=0.5\textwidth]{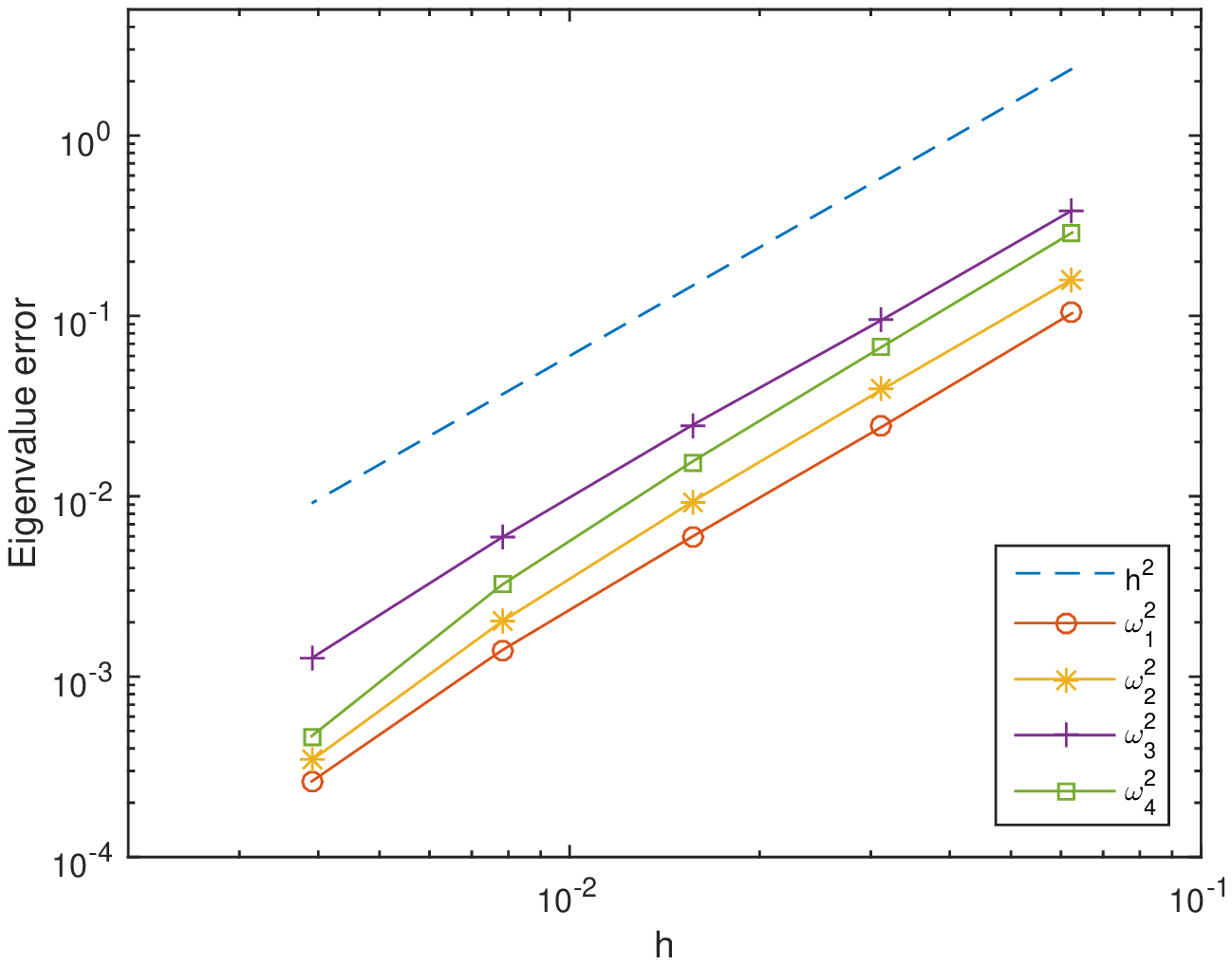}
	\caption{The log-log plots of $h$ versus the relative error of the first four eigenvalues with multiple interfaces for the case of $(\mu^-,\lambda^-) = (1,2)$, $(\mu^{+},\lambda^{+}) = (30,36)$ in Example 4. The broken line represents the optimal convergence rate.}
	\label{fig:multiple-err}
\end{figure}

\begin{figure}[!ht]
	\centering
	 \includegraphics[width=0.44\textwidth, height=0.44\textwidth]{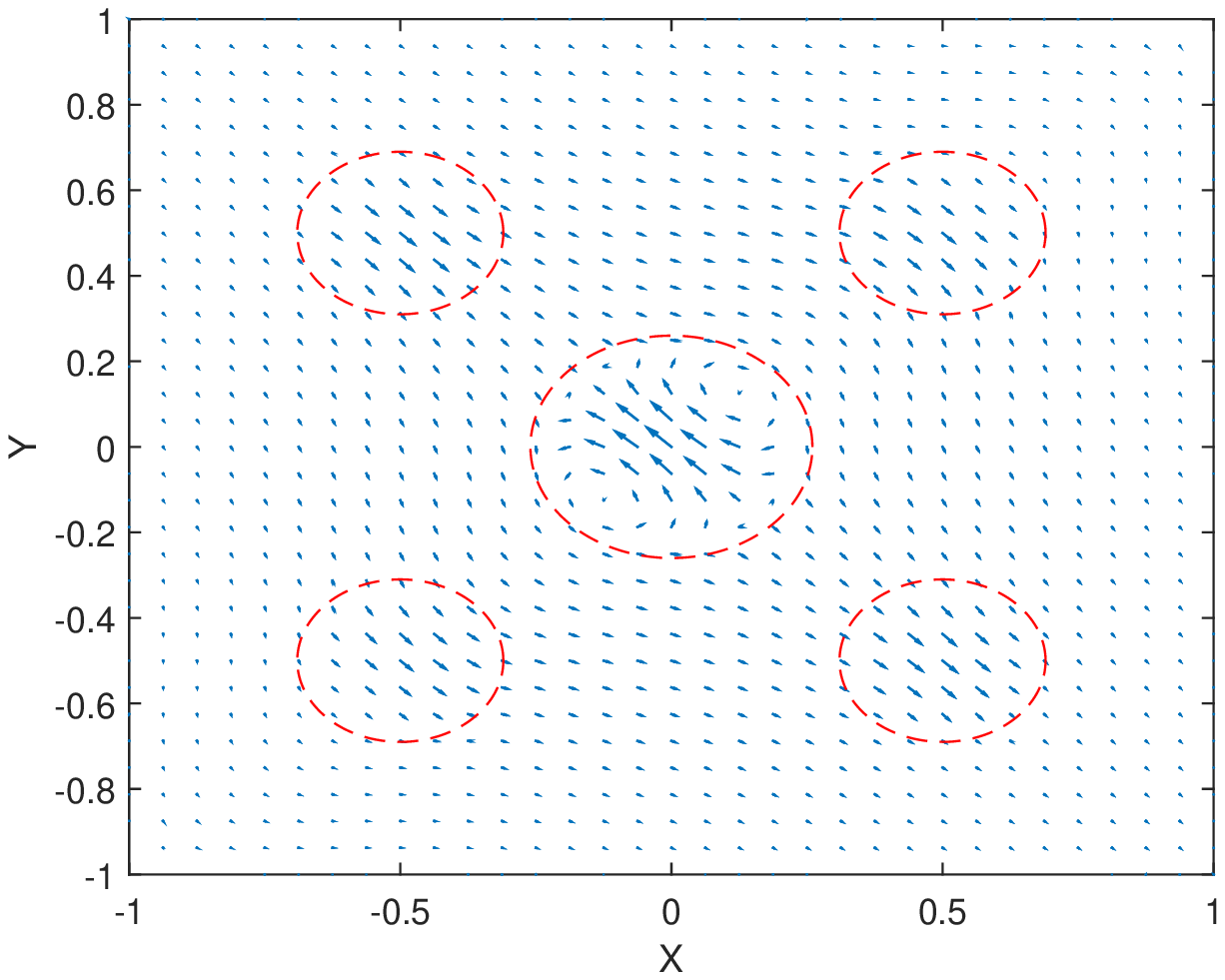} \\
	 \includegraphics[width=0.47\textwidth, height=0.43\textwidth]{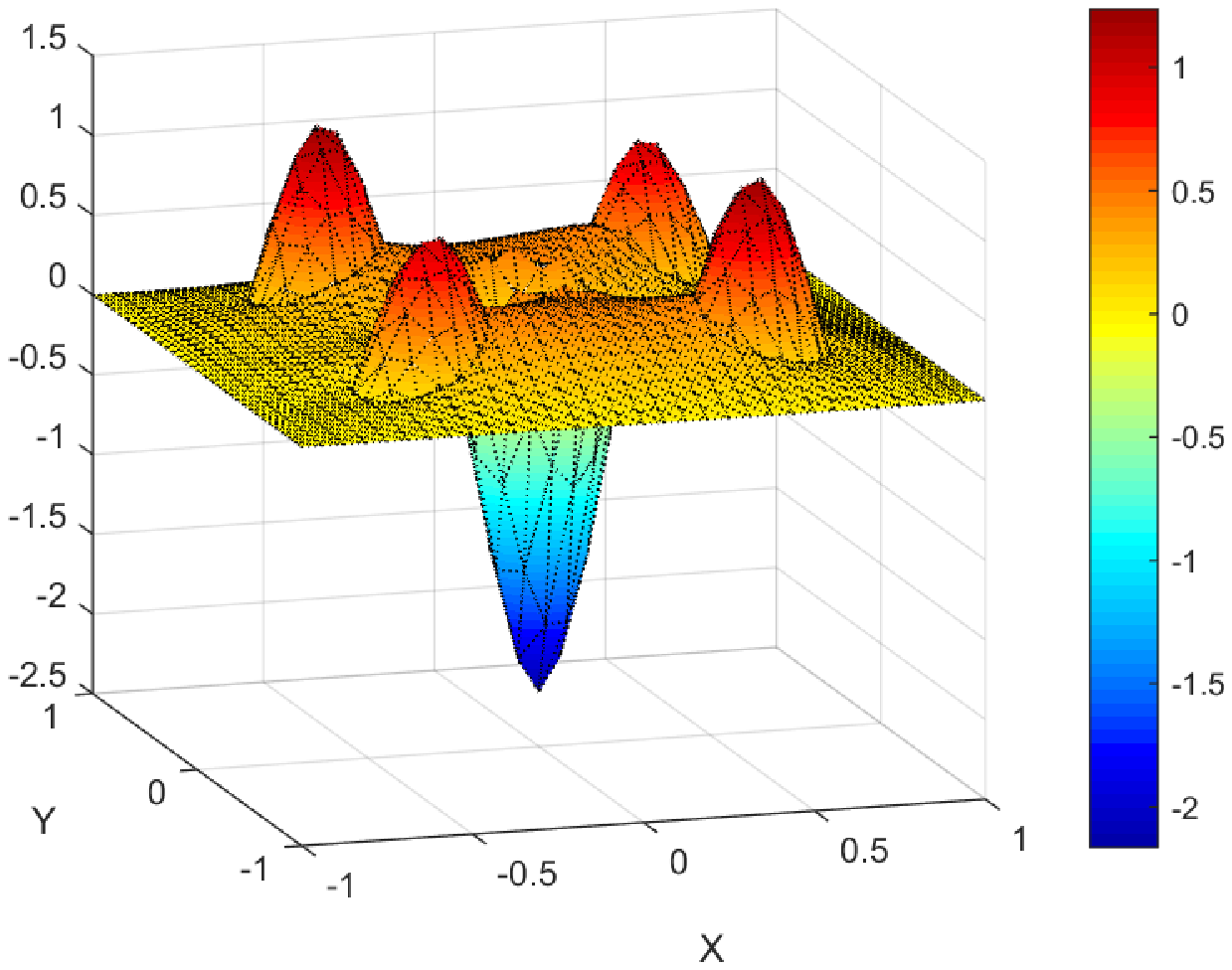}
	\includegraphics[width=0.47\textwidth, height=0.43\textwidth]{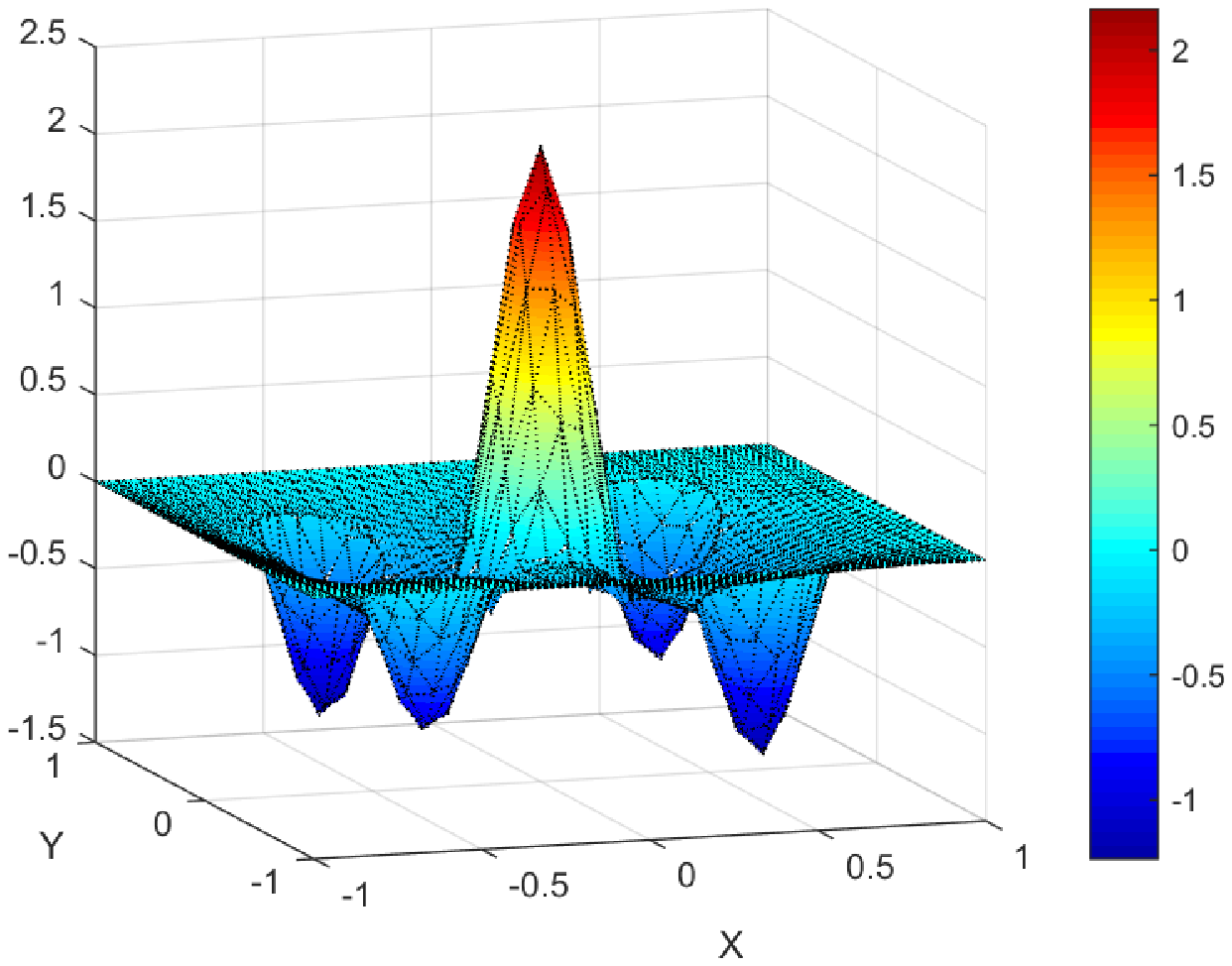}
	\caption{Eigenfunction of $\omega^2_7$ when $(\mu^-,\lambda^-) = (1,2)$, $(\mu^{+},\lambda^{+}) = (30,36)$ in Example 4 (above), x-component of eigenfunction (below on the  left), and y-component of eigenfunction (below on the right).}
	\label{fig:multiple-eigvec}
\end{figure}

\textbf{Example5} ({\it Incompressible materials}) To experiment the case of  the incompressible elastic materials, we set $(\mu^-, \mu^+) = (0.5,5),(5,0.5),\, \lambda^{\pm} = 5000\mu^{\pm}$ and $\nu^{\pm} \approx 0.4999$. We carry out similar numerical experiments with a straight-line interface to demonstrate the locking-free character of our method. The domain $\Omega$ and interface $\Gamma$ are the same as Example 3. In Figure \ref{fig:line-incomp}, we report the computed errors of first four eigenvalues by IFEM. According to Figure \ref{fig:line-incomp}, it can be seen that the method has thoroughly locking-free feature for solving the elasticity interface problems.

\begin{figure}[!ht]
	\centering
	 \includegraphics[width=0.45\textwidth, height=0.45\textwidth]{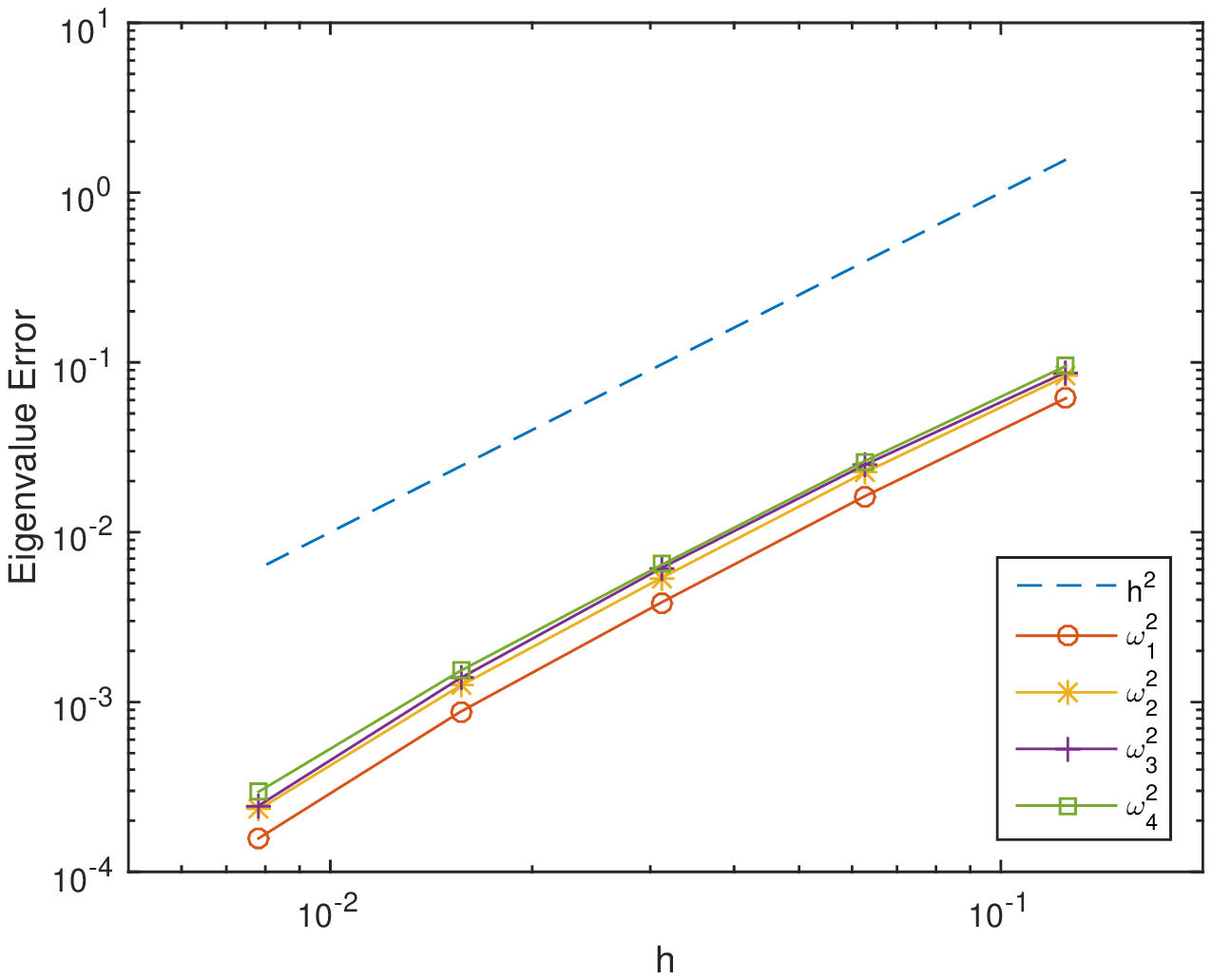}
	\includegraphics[width=0.45\textwidth, height=0.45\textwidth]{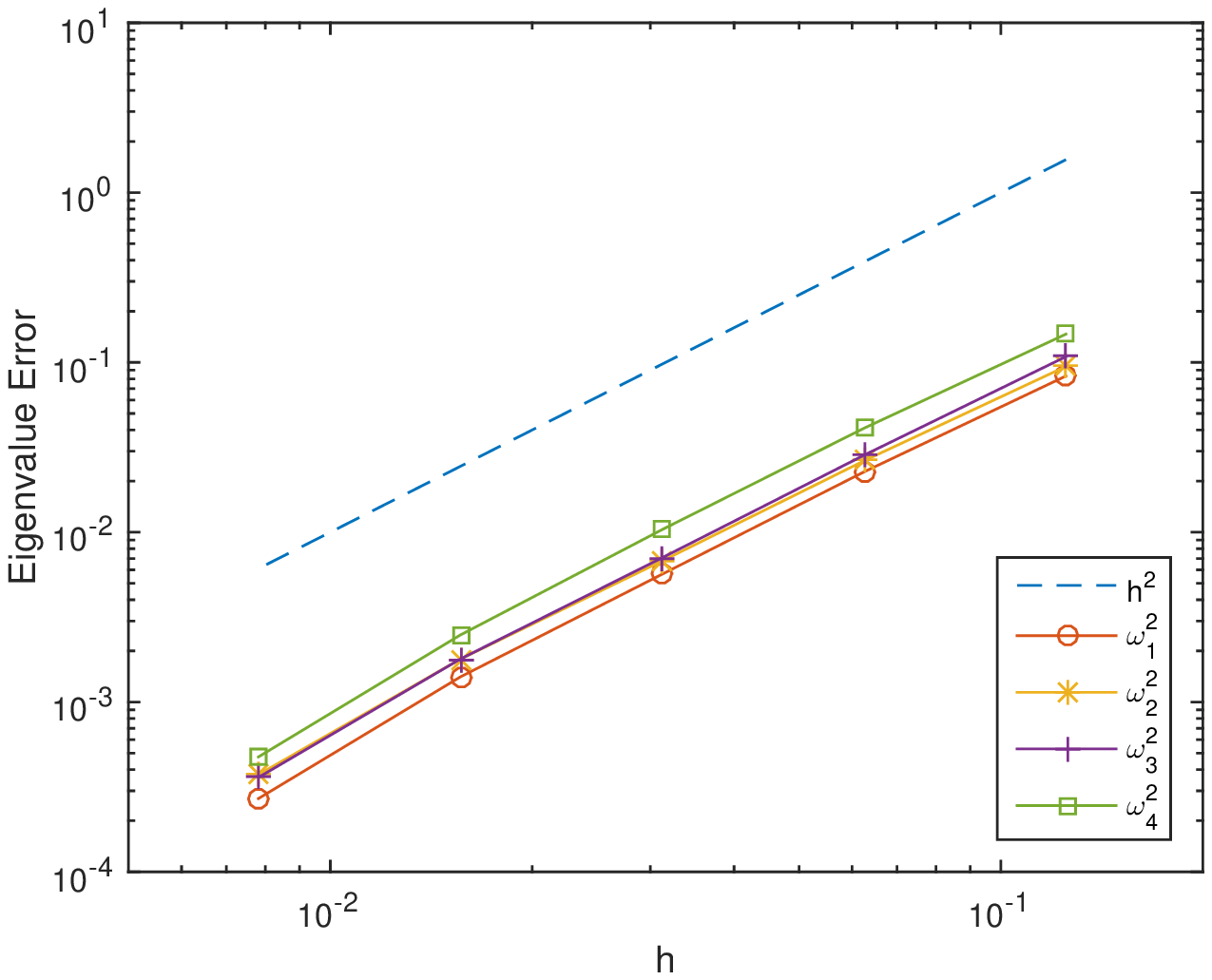}
	\caption{The log-log plots of $h$ versus the relative error of the first four eigenvalues with incompressible materials ($\nu \approx 0.4999$) for the case of $(\mu^-, \mu^+) = (0.5,5)$ (left) and $(\mu^-, \mu^+) = (5,0.5)$ (right) in Example 5. The broken line represents the optimal convergence rate.}
	\label{fig:line-incomp}
\end{figure}

\end{document}